\documentclass[a4paper, 11pt, leqno, dvipsnames]{amsart}
 
\usepackage[french, english]{babel} 
\usepackage[T1]{fontenc}
\usepackage[ansinew]{inputenc}

\date{}
\title[Unique continuation estimates for Baouendi--Grushin equations]{Unique continuation estimates for Baouendi--Grushin equations on cylinders}

\author{Paul Alphonse}
\address{(Paul Alphonse) Université de Lyon, ENSL, UMPA - UMR 5669, F-69364 Lyon}
\email{paul.alphonse@ens-lyon.fr}

\author{Albrecht Seelmann}
\address{(Albrecht Seelmann) Technische Universit\"at Dortmund, Fakult\"at f\"ur Mathematik, D-44221 Dortmund, Germany}
\email{albrecht.seelmann@tu-dortmund.de}


\makeatletter 
	\@namedef{subjclassname@2020}{\textup{2020} Mathematics Subject Classification}
\makeatother


\usepackage[top=3cm, bottom=2cm, left=3cm, right=3cm]{geometry}	

\usepackage{enumitem}

\usepackage{array}
\usepackage{amsfonts}
\usepackage{amsmath}
\usepackage{amsthm}
\usepackage{amssymb}
\usepackage{mathtools}

\usepackage{multicol}

\usepackage{bbm}

\usepackage{stmaryrd}

\usepackage[scr]{rsfso}

\usepackage{comment}

\usepackage{soul}

\usepackage[dvipsnames]{xcolor}

\usepackage{hyperref}

\hypersetup{	
colorlinks=true,
breaklinks=true,
urlcolor= RedViolet,
linkcolor= red,
citecolor=Blue
}

\numberwithin{equation}{section}

\newtheorem{thm}{Theorem}[section]
\newtheorem{prop}[thm]{Proposition}

\newtheorem{lem}[thm]{Lemma}
\newtheorem{cor}[thm]{Corollary}
\theoremstyle{definition}
\newtheorem{dfn}[thm]{Definition}
\newtheorem{ex}[thm]{Example}
\newtheorem{rk}[thm]{Remark}

\DeclareMathOperator{\Supp}{Supp}

\DeclareMathOperator{\diam}{diam}
\DeclareMathOperator{\TT}{\mathbb T}
\DeclareMathOperator{\ZZ}{\mathbb Z}

\DeclareMathOperator{\spane}{span}
\DeclareMathOperator{\thick}{\Lambda}

\DeclarePairedDelimiter{\abs}{\lvert}{\rvert}
\DeclarePairedDelimiter{\norm}{\lVert}{\rVert}
\DeclarePairedDelimiter{\sprod}{\langle}{\rangle}

\newcommand{\RR}{\mathbb{R}}
\newcommand{\NN}{\mathbb{N}}

\newcommand{\cE}{\mathcal{E}}
\newcommand{\cQ}{\mathcal{Q}}

\newcommand{\euler}{e}
\newcommand{\scrit}{R}
\newcommand{\var}{\mathbb R^d\times\mathbb T^d}

\newcommand{\bone}{\mathbf{1}}
\newcommand{\bmone}{\mathbbm{1}}
\newcommand{\dd}{\mathrm d}

\newcommand{\good}{\mathrm{gd}}
\newcommand{\dist}{d}
\newcommand{\eps}{\varepsilon}
\newcommand{\para}{\alpha}
\newcommand{\spt}{\mathbb Z^d}

\begin{document}

\begin{abstract}
	We prove time-pointwise quantitative unique continuation estimates for the evolution operators associated to
	(fractional powers of) the Baouendi--Grushin operators on the cylinder $\RR^d \times \TT^d$. Spectral
	inequalities are deduced, relating for functions from spectral subspaces associated to finite energy intervals their $L^2$-norm on the whole
	cylinder $\mathbb R^d\times\mathbb T^d$ to the $L^2$-norm on so-called thick subsets of $\mathbb R^d\times\mathbb T^d$. As a byproduct, we also obtain results on exact and approximate null-controllabilty for the corresponding evolution systems. This
	extends and complements results obtained recently by the authors and by Jaming and Wang.
\end{abstract}

\sloppy

\selectlanguage{english}

\maketitle

\section{Introduction}
The present paper continues our considerations from \cite{AlphonseS} on spectral inequalities for Shubin operators on $\RR^d$ and
extends them to Baouendi--Grushin operators on the cylinder $\RR^d \times \TT^d$, that is, to the operators
\begin{equation}\label{eq:BaouendiG}
	\Delta_{\gamma} = \Delta_x + \vert x\vert^{2\gamma}\Delta_y,\quad (x,y)\in\var,
\end{equation}
where $\gamma\geq1$ is a positive integer. Here, we consistently associate with $x$ the coordinates in $\RR^d$ and with $y$ the
coordinates in the $d$-dimensional torus $\mathbb T^d$, and $\Delta_y$ denotes the Laplacian on $\TT^d$. Analogously to
\cite{AlphonseS}, and in line with the pioneering works \cite{JL96, LR95, LZ98}, we aim at establishing so-called \emph{spectral inequalities} of the form
\begin{equation}\label{eq:geninspe}
	\forall\lambda\geq0,\ \forall f \in \cE_{\lambda}(-\Delta_{\gamma})
	,\quad
	\norm{f}_{L^2(\var)}
	\leq
	c_{\lambda,\omega} \norm{f}_{L^2(\omega)}
	,
\end{equation}
where $\omega \subset \var$ is a measurable subset with positive measure whose geometry is to be understood and
$c_{\lambda,\omega} > 0$ is a positive constant whose dependence on $\omega$ and $\lambda>0$ is to be tracked. Moreover, $\cE_\lambda(-\Delta_{\gamma}) = \bmone_{(-\infty,\lambda]}(-\Delta_{\gamma})$ stands for the spectral subspace for $-\Delta_{\gamma}$ associated with the interval $(-\infty,\lambda]$.
We refer to
\cite{AlphonseS} and the references cited therein for an overview on the notion of a spectral inequality and corresponding works
on related models. Let us mention that we consider the Baouendi--Grushin operator $\Delta_{\gamma}$ acting on $\mathbb R^d\times\mathbb T^d$, but we could also have treated the more general cylinders $\mathbb R^{d_1}\times\mathbb T^{d_2}$ with $d_1,d_2\geq1$. We decided to consider the case $d_1 = d_2$ in order to alleviate the writing.

The core of our considerations here are \emph{time-pointwise quantitative unique continuation estimates} for the evolution
operators associated to (a fractional power of) the Baouendi--Grushin operators \eqref{eq:BaouendiG}, which are of the form
\begin{equation}\label{eq:uc}
	\norm{ e^{-t(-\Delta_\gamma)^{\frac{1+\gamma}{2}}}g }_{L^2(\RR^d\times \TT^d)}^2
	\leq
	C_{\eps,\omega,t} \norm{ e^{-t(-\Delta_\gamma)^{\frac{1+\gamma}{2}}}g }_{L^2(\omega)}^2 + \eps\norm{g}_{L^2(\RR^d\times\TT^d)}^2
\end{equation}
for $t > 0$, $\eps > 0$ and $g \in L^2(\RR^d\times\TT^d)$.
The precise dependence of the unique continuation estimates \eqref{eq:uc} with respect to the
parameters $t,\varepsilon>0$ is very important since it allows to switch to the spectral parameter $\lambda>0$ and
to get the sharp general spectral estimates of the form \eqref{eq:geninspe}.
Such spectral estimates \eqref{eq:geninspe} then
imply positive exact null-controllability results for the heat-like equations associated
to the fractional power $(-\Delta_\gamma)^s$ in the regime $s>(1+\gamma)/2$. We also deduce from the estimates \eqref{eq:uc}
approximate null-controllability results in the case $s\geq(1+\gamma)/2$; the parameter $\varepsilon>0$ then plays the role of precision of approximate controls.
Our proof of such estimates \eqref{eq:uc} essentially follows the technique established in \cite{Martin-22}, which is in turn
based on \cite{Kovrijkinethesis,Kovrijkine-01}, and relies on so-called smoothing estimates for the evolution operators generated by the fractional Baouendi--Grushin operators $(-\Delta_{\gamma})^s$ when $s=(1+\gamma)/2$. The latter are obtained
by exploiting the fact that via a Fourier expansion with respect to the $y$ variable the operators $\Delta_\gamma$ are linked to the
anharmonic oscillators
\[
	-\Delta_x + \abs{n}^2\abs{x}^{2\gamma}
	,\quad
	x \in \RR^d
	,\
	n \in \ZZ^d
	.
\]
We can therefore make use of the smoothing estimates for the evolution operators associated with fractional
powers of the anharmonic oscillators $-\Delta_x + \abs{x}^{2\gamma}$ established in~\cite{Alp20b}.

\subsubsection*{Outline of the work}
In Section \ref{sec:results}, we present in detail the main results in the current work and discuss applications in the context
of control theory. Section~\ref{sec:uniqueCont} then provides a general scheme towards unique continuation estimates that is of
its own interest and hopefully proves useful also in future work on the subject. This is complemented by corresponding smoothing
estimates for the evolution operators in Section~\ref{sec:Grushin}, which serve as an input for the general scheme from
Section~\ref{sec:uniqueCont}. The proofs of the main results stated in Section~\ref{sec:results} are then collected in
Section~\ref{sec:mainProofs}. Section~\ref{sec:Shubin} demonstrates the usefulness of the general scheme from
Section~\ref{sec:uniqueCont} by complementing our work \cite{AlphonseS} on Shubin operators $(-\Delta)^m + \abs{x}^{2k}$ on
$\RR^d$, $k,m \in \NN^*$, with corresponding unique continuation estimates. In Appendices~\ref{app:local}
and~\ref{app:expoential} we finally discuss some technical results used for the proofs of our statements on (linear combinations
of) eigenfunctions and the optimality of the smoothing estimates, respectively.

\subsubsection*{Notations} The following notations and conventions will be used throughout this work:
\begin{enumerate}[label={\arabic*.},leftmargin=* ,parsep=2pt,itemsep=0pt,topsep=2pt]

	\item
	$\ZZ$ and $\NN$ denote the sets of integers and nonnegative integers, respectively. Moreover, $\NN^* = \NN\setminus\{0\}$.

	\item
	$\TT \cong \RR/(2\pi\mathbb Z)$ denotes the torus.

	\item
	The canonical Euclidean scalar product of $\mathbb R^d$ is denoted by $\cdot$, and $\vert\cdot\vert$ stands for the associated
	canonical Euclidean norm. We will also use the notation $\vert\cdot\vert_1$ for the $1$-norm of vectors in $\RR^d$.

	\item
	The length of a multi-index $\alpha=(\alpha_1,\cdots,\alpha_d)\in\mathbb N^d$ is denoted $\vert\alpha\vert$ and defined by
	\[
		\vert\alpha\vert
		=
		\abs{\alpha}_1
		=
		\alpha_1+\cdots+\alpha_d
		.
	\]

	\item
	The Lebesgue measure of a measurable set $\omega\subset\mathbb R^d$ is denoted $\vert\omega\vert$.

	\item
	$\mathbbm1_{\omega}$ denotes the characteristic function of any subset $\omega\subset\mathbb R^d$.

	\item
	For every measurable subset $\omega\subset\mathbb R^d$, the inner product of $L^2(\omega)$ is denoted
	$\langle\cdot,\cdot\rangle_{L^2(\omega)}$, while $\Vert\cdot\Vert_{L^2(\omega)}$ stands for the associated norm.

	\item
	The Fourier transform of a function $f \in L^2(\RR^d)$ is denoted $\mathscr F f$ or $\widehat{f}$, depending on the situation,
	and is given for $f \in L^1(\RR^d) \cap L^2(\RR^d)$ by
	\[
		(\mathscr F f)(\xi)
		=
		\widehat{f}(\xi)
		=
		\int_{\RR^d} f(x) e^{-i x\cdot\xi} \,\dd x
		,\quad
		\xi \in \RR^d
		.
	\]
	With this convention, Plancherel's theorem states that
	\[
		\forall f\in L^2(\mathbb R^d),\quad \Vert\widehat f\Vert_{L^2(\mathbb R^d)} = (2\pi)^{d/2}\Vert f\Vert_{L^2(\mathbb R^d)}.
	\]

	\item
	For a nonnegative selfadjoint operator $A$ on $L^2(\RR^d)$, $\cE_\lambda(A) = \bmone_{(-\infty,\lambda]}(A)$ with
	$\lambda \geq 0$ denotes the spectral subspace for $A$ associated with the interval $(-\infty,\lambda]$.

\end{enumerate}

\subsubsection*{Acknowledgments}
The first author thanks J. Bernier, J. Martin and S. Zugmeyer for many enthusiastic discussions during the preparation of this
work. The authors are very grateful to the anonymous referee for the valuable remarks which helped to improve the presentation of the paper. The second author has been partially supported by the DFG grant VE 253/10-1 entitled \textit{Quantitative unique continuation
properties of elliptic PDEs with variable 2nd order coefficients and applications in control theory, Anderson localization, and
photonics.}

\section{Statement of the main results}\label{sec:results}

This section is devoted to present the main results contained in this paper.

\subsection{Thickness in the cylinder}\label{subsubsection:thckness}

First of all, let us define the notion of thickness in the cylinder $\mathbb R^d\times\mathbb T^d$, which is central in this
work. To that end, given $L_x,L_y > 0$, we introduce in $\RR^{2d} = \RR^d \times \RR^d$ the hyperrectangle
$\thick_{L_x,L_y} := (0,L_x)^d \times (0,L_y)^d$.  The notion of thickness is then defined as follows.

\begin{dfn}\label{def:thickStrip}
	A measurable set $\omega \subset \var$ is said to be \emph{$(\theta,L_x,L_y)$-thick in $\var$} with some $\theta \in (0,1]$,
	$L_x > 0$, and $L_y \in (0,2\pi]$ if
	\[
		\forall z\in\mathbb R^{2d}
		,\quad
		\vert \omega \cap (z + \thick_{L_x,L_y}) \vert
		\geq
		\theta L_x^dL_y^d
		.
	\]
\end{dfn}

\begin{rk}
	This is a common extension of the classical notion of thickness in $\RR^{2d}$, cf., for instance, \cite{Egidi-21,
	EgidiS-21,EgidiV-20}. In this regard, let us recall that a measurable set $\omega \subset \RR^{2d}$ is called \emph{$(\theta,L)$-thick}
	with some $\theta \in (0,1]$ and $L > 0$ if
	\[
		\forall z\in\RR^{2d}
		,\quad
		\vert \omega\cap (z + (0,L)^{2d}) \vert
		\geq
		\theta L^{2d}
		.
	\]
\end{rk}

The main result in this paper is the following statement, which provides quantitative unique continuation estimates for the
evolution operators generated by the fractional Baouendi--Grushin operator $(-\Delta_{\gamma})^s$ in the critical regime
$s=(1+\gamma)/2$.

\begin{thm}\label{thm:ucthickbg}
	There exist constants $K > 0$ and $t_0\in(0,1)$, depending only on $\gamma$ and the dimension $d$, such that for every
	$(\theta,L_x,L_y)$-thick set $\omega$ in $\var$ and all $t\in(0,t_0)$, $\eps > 0$, and $g \in L^2(\var)$, we have
	\[
		\norm{ \euler^{-t(-\Delta_{\gamma})^{\frac{1+\gamma}2}} g }_{L^2(\var)}^2
		\leq
		\biggl( \frac{K}{\theta} \biggr)^{KC_{\eps,\gamma,t}}
			\norm{ \euler^{-t(-\Delta_{\gamma})^{\frac{1+\gamma}2}}g }_{L^2(\omega)}^2 + \eps\norm{g}_{L^2(\var)}^2
	\]
	with
	\[
		C_{\eps,\gamma,t}
		=
		(1 - \log\eps - \log t)\exp\biggl( \frac{KL_x}{t^{\frac{\gamma}{\gamma+1}}} + \frac{KL_y}{t} \biggr)
		.
	\]
	In the particular case of $\gamma = 1$, the same holds even for all positive times $t>0$, and the term $-\log t$ in the
	constant $C_{\eps,\gamma,t}$ can be skipped.
\end{thm}

As a corollary to Theorem~\ref{thm:ucthickbg} and its proof, we obtain the following spectral inequality.

\begin{cor}\label{cor:speinthickbg}
	Given $L_x > 0$ and $L_y \in (0,1]$, there exists a constant $K > 0$, depending only on $d$, $\gamma$, $L_x$, and $L_y$, such
	that for every $(\theta,L_x,L_y)$-thick set $\omega$ in $\var$, every $\lambda \geq 0$, and all $f \in \cE_\lambda(-\Delta_\gamma)$
	we have
	\[
		\norm{f}^2_{L^2(\var)}
		\leq
		\biggl( \frac{K}{\theta} \biggr)^{K(1+\lambda^{\frac{1+\gamma}2})} \norm{f}^2_{L^2(\omega)}
		.
	\]
	In the particular case of $\gamma = 1$, this estimate can be strengthened as
	\[
		\norm{f}_{L^2(\var)}^2
		\leq
		\biggl( \frac{K}{\theta} \biggr)^{K(1 + L_y\lambda + L_x^2\lambda)} \norm{f}_{L^2(\omega)}^2
	\]
	with a constant $K > 0$ depending only on the dimension $d$.
\end{cor}

\begin{rk}\label{rk:comparespein}
	(1)
	Considering formally $\gamma = 0$ in Corollary~\ref{cor:speinthickbg}, we are in the situation of the pure Laplacian acting on
	$\var$. Our spectral inequality is then consistent with established results for that case, see \cite{Egidi-21} and also
	\cite[Corollary~1.6]{EgidiS-21}, but without a quantitative description of the dependence on the parameters $L_x$ and $L_y$.

	(2)
	To the best of the authors knowledge, the above spectral inequalities are the first ones for the Baouendi--Grushin operator
	$\Delta_{\gamma}$ on $\mathbb R^d\times\mathbb T^d$. However, uncertainty relations for single eigenfunctions have been obtained
	in the literature for the very same operator considered on compact manifolds. Corresponding results from the work \cite{LL},
	which actually deal with more general type I H\"ormander operators, are briefly discussed in Remark~\ref{rk:compEigenfunctions}
	below.
	
	(3) The small scale limit $\theta^{-K\lambda^{(\gamma+1)/2}}$ in the spectral inequalities from Corollary \ref{cor:speinthickbg} can somehow be related to the behavior of the spectral estimates from small balls for Laplace-Beltami operators on Riemannian compact manifolds $(\mathcal M,g)$ investigated in the work \cite{LL2}. More precisely, \cite[Theorem 1.15]{LL2} states spectral estimates of the following form for functions $f\in\mathcal E_{\lambda}(-\Delta_g)$,
	\[
		\Vert f\Vert_{L^2(\mathcal M)}\le e^{(c_1\sqrt{\lambda}+c_2)(1+\log(1/r))}\Vert f\Vert_{L^2(B(x_0,r))},
	\]
	where $\Delta_g$ denotes the Laplace-Beltrami operator on $\mathcal M$ and $x_0\in\mathcal M$. Notice that the small scale limit $r^{-c_1\lambda^{1/2}}$ is parallel to the aforementioned small scale limit in Corollary~\ref{cor:speinthickbg} in the formal elliptic case $\gamma = 0$. Moreover, in the particular case $\gamma=1$, the precise small scale limit $\theta^{-K(L_y+L_x^2)\lambda}$ is probably related to the volume of subriemannian balls in the Grushin plane, but this will not be pursued here.
	
	(4)
	The fact that the spectral estimates in Corollary \ref{cor:speinthickbg} are more precise when $\gamma=1$ is due to the fact that we get pointwise unique continuation estimates in Theorem \ref{thm:ucthickbg} holding for every positive time $t>0$, which allows to track the geometric parameters $L_x>0$ and $L_y\in(0,2\pi)$ (recall that the parameter $\varepsilon>0$ allows to switch to the spectral parameter $\lambda>0$ as mentioned in the introduction).
\end{rk}

Let us now discuss the optimality of the spectral inequalities for the general Baouendi--Grushin operators $\Delta_{\gamma}$,
beginning with the geometry of the sensor set $\omega\subset\mathbb R^d\times\mathbb T^d$. We first prove that this sensor set
has to be thick in order for a spectral inequality to hold.

\begin{prop}\label{prop:geometry}
	Let $\omega\subset\mathbb R^d\times\mathbb T^d$ be a measurable set with positive measure and $\lambda\geq0$ be a non-negative
	energy level. If for some constant $c_{\lambda,\omega}>0$ we have the spectral inequality
	\[
		\forall f\in\mathcal E_{\lambda}(-\Delta_{\gamma}),\quad\Vert f\Vert_{L^2(\mathbb R^d\times\mathbb T^d)}
		\le
		c_{\lambda,\omega}\Vert f\Vert_{L^2(\omega)}
		,
	\]
	then the set $\omega$ is thick in $\mathbb R^d\times\mathbb T^d$.
\end{prop}

Let us now focus on the power in the spectral inequalities stated in Corollary \ref{cor:speinthickbg}. The following result,
which is exactly the counterpart of \cite[Proposition 1.12]{LL} for Baouendi--Grushin operators acting on the cylinder $\var$
(and not on a compact manifold as in the work \cite{LL}), shows that the power $(1+\gamma)/2$ appearing in
Corollary~\ref{cor:speinthickbg} is optimal for sensor sets $\omega\subset\var$ avoiding the critical line $\{x = 0\}$.

\begin{prop}\label{prop:power}
	Let $\omega\subset\mathbb R^d\times\mathbb T^d$ be a measurable set satisfying the geometric condition
	$\overline\omega\cap\{x=0\} = \emptyset$. Then, there exist positive constants $c_0,c_1>0$ and a sequence
	$((\lambda_n,\psi_n))_n$ of eigenpairs of the high-order Baouendi--Grushin operator $-\Delta_{\gamma}$ with
	$\lambda_n\rightarrow+\infty$ such that for each $n$ we have
	\[
		\Vert\psi_n\Vert_{L^2(\mathbb R^d\times\mathbb T^d)}\geq c_0e^{c_1\lambda_n^{\frac{1+\gamma}2}}\Vert\psi_n\Vert_{L^2(\omega)}.
	\]
\end{prop}

\subsection{Applications to control theory}

We now apply the unique continuation estimates from Theorem \ref{thm:ucthickbg} and the spectral inequalities from
Corollary~\ref{cor:speinthickbg} to the study of the null-controllability of the evolution equations associated the fractional
Baouendi--Grushin operators. More precisely, given a positive real number $s>0$, we consider the heat-like evolution equation
\begin{equation}\label{eq:contbg}\tag{$E_{\gamma,s}$}
	\left\{\begin{aligned}
		& \partial_tf(t,x,y) + (-\Delta_{\gamma})^sf(t,x,y) = h(t,x,y)\mathbbm 1_{\omega}(x,y),\quad t>0,\ (x,y)\in \mathbb R^d\times\mathbb T^d, \\
		& f(0,\cdot,\cdot) = f_0\in L^2(\mathbb R^d\times\mathbb T^d),
	\end{aligned}\right.
\end{equation}
where $\omega\subset\mathbb R^d\times\mathbb T^d$ is a control support and $h\in L^2((0,T)\times\omega)$ is a control. We focus
on two notions of null-controllability and present a result for each of them. The results presented in this subsection are in
line with articles devoted to the study of the null-controllability of Baouendi--Grushin heat equations. A pioneering article in
this theory is \cite{BCG}, which paved the way for a numerous series of articles of which we can cite
\cite{ABM, BDE, BMM, DKR, DK, JY, Ko}.

\subsubsection{Exact null-controllability}
We first study the exact null-controllability properties of the equation \eqref{eq:contbg} and begin with recalling the
corresponding notion.

\begin{dfn}\label{dfn:exactcontrol}
	The evolution equation \eqref{eq:contbg} is said to be \emph{exactly null-controllable from the control support $\omega$ in time
	$T > 0$} if for every initial datum $f_0\in L^2(\mathbb R^d\times\mathbb T^d)$ there exists a control function
	$h\in L^2((0,T)\times\mathbb R^d\times\mathbb T^d)$ such that the mild solution to \eqref{eq:contbg} satisfies $f(T,\cdot) = 0$.
\end{dfn}

According to the Hilbert Uniqueness Method, the exact null-controllability of the equation \eqref{eq:contbg} from
$\omega\subset\var$ at time $T>0$ is equivalent to the existence of a positive constant $C_{\omega,T}>0$ such that for all
$g\in L^2(\var)$,
\begin{equation}\label{eq:exactobs}
	\big\Vert e^{-T(-\Delta_{\gamma})^s}g\big\Vert^2_{L^2(\mathbb R^d\times\mathbb T^d)}
	\le C_{\omega,T}\int_0^T\big\Vert e^{-t(-\Delta_{\gamma})^s}g\big\Vert^2_{L^2(\omega)}\,\mathrm dt.
\end{equation}
The latter is called an exact observability estimate for the semigroup $(e^{-t(-\Delta_{\gamma})^s})_{t\geq0}$. Several results
in the literature allow to derive exact observability estimates from spectral inequalities. Most of them are proven by using the
well-known Lebeau--Robbiano strategy, which originated from the work \cite{LR95}. We do not give a precise statement of one of those results here, but just mention that the
estimate \eqref{eq:exactobs} holds once a spectral inequality of the following form has been established, where $d_0>0$, $d_1\geq0$
and $\eta\in(0,s)$:
\[
	\forall f\in\mathcal E_{\lambda}(-\Delta_{\gamma}),\quad \Vert f\Vert_{L^2(\mathbb R^d\times\mathbb T^d)}
	\le
	d_0 e^{d_1\lambda^{\eta}}\Vert f\Vert_{L^2(\omega)}.
\]
We refer the reader to \cite[Theorem 2.8]{NakicTTV-20} for a precise statement particularly well adapted towards an explicit form
of the constant $C_{\omega,T}$ in terms of the parameters $d_0$, $d_1$, and $\eta$. We thus deduce from
Corollary~\ref{cor:speinthickbg} the following exact null-controllability result in the high dissipation regime $s>(1+\gamma)/2$.

\begin{thm}\label{thm:encbg}
	Suppose that $s > (1+\gamma)/2$, and let $T > 0$ and $\omega\subset\mathbb R^d\times\mathbb T^d$ be a measurable set with
	positive measure. The following assertions are equivalent:
	\begin{enumerate}
		\item[$(i)$] The equation \eqref{eq:contbg} is exactly null-controllable from $\omega$ in time $T$.
		\item[$(ii)$] The set $\omega$ is thick in $\mathbb R^d\times\mathbb T^d$.
	\end{enumerate}
\end{thm}

\begin{rk}\label{ref:encbg}
	(1)
	The implication $(i)\Rightarrow(ii)$ has already been proven in \cite[Proposition~2.14]{AlphonseS}. Moreover, the converse
	$(ii)\Rightarrow(i)$ strengthens Theorem~2.15 in \cite{AlphonseS}, which states that for every thick set $\omega \subset \RR^d$
	(cf.\ Remark~\ref{rk:geometry}), the equation \eqref{eq:exactobs} is exactly null-controllable in every positive time $T>0$ from
	the sensor set $\omega\times\mathbb T^d$. Note that the set $\omega\times\mathbb T^d$ is then thick in
	$\mathbb R^d\times\mathbb T^d$ in the sense of Definition~\ref{def:thickStrip}.

	(2)
	The exact null-controllability properties of the equation \eqref{eq:contbg} when $0<s\le(1+\gamma)/2$ were investigated in
	\cite[Theorem 2.17, Theorem 2.19]{AlphonseS}. Let us briefly recall the content of those results, starting with the case where
	a control support $\omega\subset\mathbb R^d\times\mathbb T^d$ satisfying the geometric condition
	$\overline\omega\cap\{x=0\} = \emptyset$ is considered. On the one hand, the equation \eqref{eq:contbg} is then never exactly
	null-controllable from $\omega$ when $0<s<(1+\gamma)/2$. On the other hand, in the critical case $s=(1+\gamma)/2$, the equation
	\eqref{eq:contbg} is not exactly null-controllable from $\omega$ at any time $T<T_*$, where $T_*>0$ is a positive time
	explicitly given in terms of $\gamma$ and the distance from the origin to $\omega$. Moreover, still in the situation
	$s=(1+\gamma)/2$, the equation \eqref{eq:contbg} is exactly null-controllable in every large enough positive time $T > T_*$ from
	every control support of the form $\omega\times\mathbb T^d$ with a thick set $\omega \subset \RR^d$.
\end{rk}

\begin{rk}
The recent work \cite{JY} studies exact null-controllability for different classes of
heat-like equations, including ones associated to operators of the same form as \eqref{eq:BaouendiG}, but with $\abs{x}^{2\gamma}$ replaced by a
more general potential $V(x)$, that is, for the evolution equations associated with the operators $(-\Delta_V)^s$ with
\[
	\Delta_V = \Delta_x + V(x)\Delta_y,\quad(x,y)\in\mathbb R^d\times\mathbb T^d.
\]
More precisely, the results in that paper apply in particular with potentials $V\in L^{\infty}_{\mathrm{loc}}(\mathbb R^d)$ satisfying the following growth conditions
\[
	c_1\vert x\vert^{\beta_1}\le V(x)\quad\text{and}\quad\vert V(x)\vert + \vert\nabla V(x)\vert\le c_2\vert x\vert^{\beta_2},
\]
with $c_1,c_2,\beta_1,\beta_2>0$. For this class of potentials, including $V(x) = \vert x\vert^{2\beta}$ with $\beta>1/2$ a positive number (and not necessarily an integer), \cite[Theorem 1.9]{JY} states an exact observability result in any positive time $T>0$ for the fractional Baouendi--Grushin equations associated with the operators $(-\Delta_V)^s$ when $s>(\beta+1)/2$. Let us mention that the very same result also applies for a wider class of potentials, including $V(x) = \vert x\vert^{2\beta}$ with $0<\beta\le 1/2$ for which the same result holds when $s>(2\beta+2)/3$. Regarding the class of potentials considered, \cite[Theorem 1.9]{JY} is therefore more general than Theorem \ref{thm:encbg} above. However, the control supports considered in \cite[Theorem 1.9]{JY} take the tensorized form $\omega'\times\mathbb T^d$, with $\omega'\subset\mathbb R^d$, and are moreover assumed to have suitably distributed inner points. The latter condition is a stronger notion than being thick, see \cite[Definition 1.5]{JY}. In the particular case of the monomial potentials $\vert x\vert^{2\gamma}$, Theorem \ref{thm:encbg} above allows to consider more general thick control supports $\omega\subset\mathbb R^d\times\mathbb T^d$. 
\end{rk}

\subsubsection{Cost-uniform approximate null-controllability}\label{subsubsec:approx}
We are now interested in the concept of approximate null-controllability with uniform cost for the equation \eqref{eq:contbg},
which is defined as follows.

\begin{dfn}\label{dfn:approxunifcost}
	The evolution equation \eqref{eq:contbg} is said to be \textit{approximately null-controllable with uniform cost} from the
	control support $\omega$ in time $T>0$ if for all $\varepsilon>0$, there exists a positive constant $C_{\varepsilon,s,\omega,T}>0$
	such that for all $f_0\in L^2(\mathbb R^d\times\mathbb T^d)$, there exists a control $h\in L^2((0,T)\times\omega)$ such that the
	mild solution of \eqref{eq:contbg} satisfies 
	\[
		\Vert f(T,\cdot)\Vert_{L^2(\mathbb R^d\times\mathbb T^d)}\le\varepsilon\Vert f_0\Vert_{L^2(\mathbb R^d\times\mathbb T^d)},
	\]
	with moreover
	\begin{equation}\label{eq:contcost}
		\frac1T\int_0^T\Vert h(t,\cdot)\Vert^2_{L^2(\omega)}\, \mathrm dt
		\le
		C_{\varepsilon,s,\omega,T}\Vert f_0\Vert^2_{L^2(\mathbb R^d\times\mathbb T^d)}.
	\end{equation}
\end{dfn}

Just as the Hilbert Uniqueness Method gives a dual interpretation of the notion of exact null-controllability in terms of exact
observability inequalities, there is also a dual interpretation of the notion of cost-uniform approximate null-controllability in
terms of weak observability inequalities. This is the purpose of the following result, taken from the work \cite{TWX}.

\begin{prop}[Proposition~6 in \cite{TWX}]\label{prop:dualityapproxcontuc}
	The evolution equation \eqref{eq:contbg} is cost-uniformly approximately null-controllable from the control support $\omega$ in
	time $T>0$ if for all $\varepsilon\in(0,1)$, there exists a positive constant $C_{\varepsilon,s,\omega,T}>0$ such that for all
	$g\in L^2(\mathbb R^d\times\mathbb T^d)$,
	\begin{equation}\label{eq:ucapproxcont}
		\big\Vert e^{-T(-\Delta_{\gamma})^s}g\big\Vert^2_{L^2(\mathbb R^d\times\mathbb T^d)}
		\le
		\frac{C_{\varepsilon,s,\omega,T}}T\int_0^T\big\Vert e^{-t(-\Delta_{\gamma})^s}g\big\Vert^2_{L^2(\omega)}\,\mathrm dt +
			\varepsilon\Vert g\Vert^2_{L^2(\mathbb R^d\times\mathbb T^d)}.
	\end{equation}
	Moreover, every constant $C_{\varepsilon,s,\omega,T}>0$ satisfying \eqref{eq:ucapproxcont} also satisfies \eqref{eq:contcost}
	and vice versa.
\end{prop}

Our result of cost-uniform approximate null-controllability for the equation \eqref{eq:contbg} is the following one.

\begin{thm}\label{thm:cuancbg}
	Suppose that $s\geq(1+\gamma)/2$, and let $T > 0$ and $\omega\subset\mathbb R^d\times\mathbb T^d$ be a measurable set with
	positive measure. The following assertions are equivalent:
	\begin{enumerate}
		\item[$(i)$] The equation \eqref{eq:contbg} is cost-uniformly approximately null-controllable from $\omega$ in time $T$.
		\item[$(ii)$] The set $\omega$ is thick in $\mathbb R^d\times\mathbb T^d$.
	\end{enumerate}
	Moreover, in the critical dissipation setting $s=(1+\gamma)/2$, for every $(\theta,L_x,L_y)$-thick set
	$\omega\subset\mathbb R^d\times\mathbb T^d$ with $\theta\in(0,1]$ and $L>0$, the control cost $C_{\varepsilon,s,\omega,T}>0$
	appearing in \eqref{eq:contcost} satisfies
	\[
		C_{\varepsilon,\gamma,(1+\gamma)/2,\omega,T}
		\le\bigg(\frac K{\theta}\bigg)^{K(1-\log\varepsilon - \log T)\exp(KL_x/T^{\frac\gamma{1+\gamma}} + KL_y/T)},
	\]
	where $K>0$ is a positive constant only depending on $\gamma$ and the dimension $d$. In the particular case of $\gamma = 1$,
	the term $-\log T$ in the exponent of the right-hand side of the latter inequality can be skipped.
\end{thm}

\begin{rk}
	(1) Notice that the critical case $s = s_{\gamma}=(1+\gamma)/2$ is allowed in Theorem~\ref{thm:cuancbg}, whereas it is excluded in
	Theorem~\ref{thm:encbg}, which is not surprising given the results recalled in Remark~\ref{ref:encbg}. More precisely, it is
	known from \cite[Theorem 2.17]{AlphonseS} that the equation (\hyperref[eq:contbg]{$E_{\gamma,s_{\gamma}}$}) is not exactly
	null-controllable in every positive time $T>0$ from thick control supports $\omega\subset\var$ in general. However,
	Theorem~\ref{thm:cuancbg} shows that the equation (\hyperref[eq:contbg]{$E_{\gamma,s_{\gamma}}$}) is cost-uniformly
	approximately null-controllable from every thick control support $\omega\subset\var$ and in every positive time $T>0$, with a
	control cost $C_{\varepsilon,s_{\gamma},\omega,T}$ enjoying a polynomial behavior with respect to the parameter $\varepsilon>0$.
	
	(2) Similar approximate observability inequalities with polynomial costs have been obtained in \cite[Theorem 1.22]{LL} for general $2$-step type I H\"ormander operators $\mathcal L$ (the operator $\Delta_{\gamma}$ with $\gamma=1$ being an example) acting on compact manifolds $\mathcal M$. They are valid only for large positive times $T>T_0+\eta$, where $T_0>0$ is fixed and $\eta>0$ is any positive number, and they take the general form
	\[
		D\big\Vert e^{-T\mathcal L}g\big\Vert^2_{L^2(\mathcal M)}
		\le
		\frac1{\varepsilon^{\beta}}\int_{T-\eta}^T\big\Vert e^{-t\mathcal L}g\big\Vert^2_{L^2(\omega)}\,\mathrm dt +
		\varepsilon\Vert g\Vert^2_{L^2(\mathcal M)},
	\]
	where $\omega\subset\mathcal M$ is an open set, $D = \min(\frac{e^{-C/\eta}}{C},1)$ and $\beta = \frac{T_0}{T-(T_0+\eta)}$. The merit of Theorem~2.13 is to prove analogous approximate observability estimates with polynomial costs in any positive times $T>0$ for the Baouendi--Grushin operators $(-\Delta_{\gamma})^s$ (including the 2-step case $\gamma=1$) in the non-compact setting of the manifold $\mathbb R^d\times\mathbb T^d$. On the other hand, \cite[Theorem 1.22]{LL} allows to consider a wider class of $2$-step operators, and is valid in a very general geometric setting.

	(3) Investigating the cost-uniform approximate null-controllability properties of the equation \eqref{eq:contbg} in the regime
	$0<s<(1+\gamma)/2$ is an open and very interesting problem that will not be tackled in the present paper.
\end{rk}

\begin{rk}
	Let us briefly explain how Theorem~\ref{thm:cuancbg} can be deduced from all the previous results.

	(1)
	The implication $(ii)\Rightarrow(i)$ in the case $s>(1+\gamma)/2$ is a consequence of Theorem~\ref{thm:encbg} since the notion
	of exact null-controllability is stronger than the notion of cost-uniform approximate null-controllability.

	(2)
	The implication $(ii)\Rightarrow(i)$ in the case $s=(1+\gamma)/2$ is obtained from the above
	Proposition~\ref{prop:dualityapproxcontuc} by integrating the quantitative unique continuation properties in
	Theorem~\ref{thm:ucthickbg} with respect to $t$. Indeed, since the evolution operators $e^{-t(-\Delta_{\gamma})^s}$ are
	contractions on $L^2(\var)$, we have 
	\begin{align*}
		\big\Vert e^{-T(-\Delta_{\gamma})^s}g\big\Vert^2_{L^2(\var)} 
		& = \frac1T\int_0^T\big\Vert e^{-T(-\Delta_{\gamma})^s}g\big\Vert^2_{L^2(\var)}\,\mathrm dt \\[5pt]
		& \le \frac1T\int_0^T\big\Vert e^{-t(-\Delta_{\gamma})^s}g\big\Vert^2_{L^2(\var)}\,\mathrm dt.
	\end{align*}

	(3)
	The implication $(i)\Rightarrow(ii)$ can be proved by proceeding similarly as in \cite[Theorem~2.1\,(i)]{AM}. We omit the
	details here.
\end{rk}

\subsection{Eigenfunctions}
Let us finally present uncertainty relations for (linear combinations of) eigenfunctions of the operator $\Delta_{\gamma}$. We
refer to Section~\ref{subsubsec:eigenfunctions} where a precise description of this countable family of eigenfunctions and the
associated eigenvalues is given. The first result concerns single eigenfunctions.

\begin{thm}\label{thm:ucsingleigenfunction}
	Let $\omega \subset \RR^d \times \TT^d$ be a measurable set with positive measure. Then, there exists a positive constant $K>0$,
	depending only on $\omega$, $\gamma$, and the dimension $d$, such that for every eigenfunction $\psi\in L^2(\var)$ of the operator
	$\Delta_{\gamma}$ with corresponding eigenvalue $\lambda > 0$, we have
	\[
		\norm{\psi}^2_{L^2(\var)}
		\leq
		K\euler^{ K \lambda^{\frac{1+\gamma}2}\log(1+\lambda) } \norm{\psi}^2_{L^2(\omega)}
		.
	\]
\end{thm}

\begin{rk}\label{rk:compEigenfunctions}
As already mentioned in Remark~\ref{rk:comparespein}, uncertainty relations were obtained for the eigenfunctions of the
Baouendi--Grushin operator $\Delta_{\gamma}$ acting on a compact manifold $\mathcal M$ (rigorously, of a variant of this operator
with analytic coefficients). More precisely, given an open subset $\omega\subset\mathcal M$, we get from \cite[Theorem~1.10]{LL}
the uncertainty relation
\[
	\norm{\psi}^2_{L^2(\mathcal M)}
	\le 
	K\euler^{ K \lambda^{\frac{1+\gamma}2}} \norm{\psi}^2_{L^2(\omega)}
	,
\]
where $\psi$ is an eigenfunction of the operator $\Delta_{\gamma}$ associated with the eigenvalue $\lambda>0$. In fact, the
result \cite[Theorem 1.10]{LL} holds more generally with type I H\"ormander operators. It is worth to note that the power
$(1+\gamma)/2$ here is essentially consistent with the one stated in Theorem~\ref{thm:ucsingleigenfunction}. An interpretation of
the additional $\log(1+\lambda)$ term in Theorem~\ref{thm:ucsingleigenfunction} is given in Remark~\ref{rk:geometry} below.
\end{rk}

Our second result in this section concerns linear combinations of eigenfunctions of the operator $\Delta_{\gamma}$. For any
non-negative energy level $\lambda\geq0$, let us consider the subspace
\begin{equation}\label{eq:combilineigen}
	E_{\lambda}(-\Delta_{\gamma})
	=
	\spane\big\{\psi\in \cE_{\lambda}(-\Delta_{\gamma}) \colon -\Delta_{\gamma}\psi = \mu\psi,\,\mu\in[0,\lambda]\big\}
	\subsetneq
	\cE_{\lambda}(-\Delta_{\gamma})
	.
\end{equation}
Notice that the spectral subspace $\mathcal E_{\lambda}(-\Delta_{\gamma})$ does not coincide with the above subspace
$E_{\lambda}(-\Delta_{\gamma})$ since the latter is spanned by proper eigenfunctions (described in Section \ref{subsubsec:eigenfunctions}), whereas the whole spectrum of the operator $-\Delta_{\gamma}$ is equal to $[0,+\infty)$ and thus has a continuous part. Details about the spectral theory of the operator $-\Delta_{\gamma}$ are given in Section~\ref{subsec:pro}. We obtain the
following uncertainty relations for functions in $E_{\lambda}(-\Delta_{\gamma})$.

\begin{thm}\label{thm:ucgeneraleigenfunctions}
	Let $\omega \subset \RR^d \times \TT^d$ be a measurable set with positive measure. Then, there exists a positive constant $K>0$,
	depending only on $\omega$, $\gamma$, and the dimension $d$, such that for all $\lambda \geq 0$ and all
	$f \in E_{\lambda}(-\Delta_{\gamma})$,
	\[
		\norm{f}^2_{L^2(\var)}
		\leq
		K\euler^{ K \lambda^{(\frac12+\frac1{2\gamma})(1+\gamma)}\log(1+\lambda) } \norm{f}^2_{L^2(\omega)}
		.
	\]
\end{thm}

\begin{rk}\label{rk:geometry}
	Notice that Theorem~\ref{thm:ucsingleigenfunction} and Theorem~\ref{thm:ucgeneraleigenfunctions} allow to consider sensor sets
	$\omega\subset\var$ with merely positive measure, whereas Corollary~\ref{cor:speinthickbg} requires the sensor set to be thick
	in the cylinder $\mathbb R^d\times\mathbb T^d$. This is due to the fact that the eigenfunctions of the operator
	$\Delta_{\gamma}$ enjoy localizing properties, in contrast to general functions in the spectral subspaces
	$\mathcal E_{\lambda}(-\Delta_{\gamma})$. More precisely, as becomes clear in the proof of
	Theorem~\ref{thm:ucgeneraleigenfunctions}, there is a constant $c > 0$, depending only on $\gamma$ and the dimension $d$, such
	that the functions $f\in E_{\lambda}(-\Delta_{\gamma})$ are localized in the sense
	\[
		\norm{f}_{L^2(\var)}
		\le
		\norm{f}_{L^2(B(0,c\lambda^{1/2\gamma})\times\mathbb T^d)}
		.
	\]
	It is therefore sufficient to prove for functions in $E_{\lambda}(-\Delta_{\gamma})$ estimates on the finite cylinder
	$B(0,c\lambda^{1/2\gamma})\times\mathbb T^d$ in order to obtain estimates on the whole cylinder $\var$. Moreover, as discussed
	in Section~\ref{subsec:eigenfunctionsproofs}, the $\log(1+\lambda)$ term appearing in Theorems~\ref{thm:ucsingleigenfunction}
	and~\ref{thm:ucgeneraleigenfunctions} is linked to the volume of this finite cylinder
	$B(0,c\lambda^{1/2\gamma})\times\mathbb T^d$.
\end{rk}

\begin{rk}
	(1)
	Notice that the two powers in Theorems~\ref{thm:ucsingleigenfunction} and~\ref{thm:ucgeneraleigenfunctions} are different, and
	that the one for single eigenfunctions in Theorem~\ref{thm:ucsingleigenfunction} is stronger. This suggests that there are
	cancellation phenomena, just as we already know for linear combinations of eigenfunctions of the Laplacian on the torus, see,
	e.g., Turan's lemma in \cite{FM,N}.

	(2)
	Let us now analyse the power in the spectral inequalities of Theorem~\ref{thm:ucgeneraleigenfunctions}. On the one hand, the
	term $1/2+1/(2\gamma)$ is familiar from the literature since it also plays a role in the spectral inequalities for linear
	combinations of eigenfunctions of the anharmonic oscillator $H_{\gamma} = -\Delta + \vert x\vert^{2\gamma}$, which according
	to \cite[Theorem~2.1\,$(ii)$]{M} read
	\[
		\norm{f}^2_{L^2(\RR^d)}
		\leq
		Ke^{K\lambda^{\frac12+\frac1{2\gamma}}\log(1+\lambda)} \norm{f}^2_{L^2(\omega)}
		,\quad
		f \in \cE_\lambda(H_{\gamma})
		,\
		\lambda>0
		.
	\]
	Here, the operator $H_\gamma$ is linked to the Baouendi--Grushin operator by a diagonalization with respect to the $y$ variable,
	see the similarity relation \eqref{eq:sim} below. Moreover, the term $1/2+1/(2\gamma)$ can be interpreted as the (inverse of
	the) order of the anharmonic oscillator $H_{\gamma}$ in a class of pseudo-differential operators associated with a H\"ormander
	metric, as proved in the paper \cite[Corollary~3.8]{CDR}. On the other hand, the term $1+\gamma$ is the hypoellipticity index
	of the Baouendi--Grushin operator $\Delta_{\gamma}$, see, e.g., \cite[Example~1.7]{LL}.
\end{rk}

\begin{rk}
	Let us consider the Baouendi--Grushin operator 
	\[
		\Delta^{bd}_{\gamma} = \Delta_x + \vert x\vert^{2\gamma}\Delta_y,\quad (x,y)\in\mathbb R^d\times(0,2\pi)^d,
	\]
	where $\Delta_y$ denotes in this case the Laplacian on the hypercube $(0,2\pi)^d$ with Dirichlet boundary conditions. In contrast to the operator $\Delta_{\gamma}$ whose spectrum has a continuous part, the
	spectrum of the operator $\Delta^{bd}_{\gamma}$ is now purely discrete, see Section~\ref{subsec:bd}. As a consequence, the
	spectral subspaces $\mathcal E_{\lambda}(-\Delta^{bd}_{\gamma})$ consist exclusively of linear combinations of eigenfunctions
	of the operator $\Delta^{bd}_{\gamma}$, that is, the spaces $\mathcal E_{\lambda}(-\Delta^{bd}_{\gamma})$ and
	$E_{\lambda}(-\Delta^{bd}_{\gamma})$ coincide, which is not the case for the operator $\Delta_{\gamma}$. Corresponding variants of Theorems~\ref{thm:ucsingleigenfunction}
	and~\ref{thm:ucgeneraleigenfunctions} are still valid for the operator $\Delta^{bd}_{\gamma}$ as explained in
	Section~\ref{subsec:bd}, and Theorem \ref{thm:ucgeneraleigenfunctions} therefore deals with the complete spectral subspaces $\mathcal E_{\lambda}(-\Delta^{bd}_{\gamma})$.
	This has ramifications on the null-controllability properties of the evolution equations associated to the operator
	$\Delta^{bd}_{\gamma}$, but we do not pursue this further here. 
\end{rk}

\section{General unique continuation estimates}\label{sec:uniqueCont}
In this section, we provide a general framework towards unique continuation estimates for functions satisfying suitable smoothing
estimates. Here, we follow the general strategy from the classical approach by Kovrijkine \cite{Kovrijkinethesis,Kovrijkine-01},
which has been exploited, adapted, and generalized in several recent works such as
\cite{AlphonseS, BeauchardJPS-21, DickeSV, DickeSV-23, EgidiS-21, EgidiV-20, Martin-22, MartinPS-22, WangWZZ-19}.
We focus specifically on \cite{Martin-22} and aim for a ready-made tool box to be applied in various different situations. This
tool box will be formulated explicitly for domains in $\RR^d$, but it applies, of course, also to domains in $\RR^{2d}$ like the
strip $\RR^d \times (0,2\pi)^d$, which corresponds to the cylinder $\var$ from Section~\ref{sec:results}.
The strategy itself relies on a suitable covering of the underlying domain and a local estimate on a sufficiently large subclass
of the covering sets, so-called good elements of the covering. The local estimate is presented in
Section~\ref{ssec:local}, whereas the covering strategy is discussed in Sections~\ref{ssec:good} and~\ref{ssec:covering}. The
general scheme is summarized in Corollary~\ref{cor:goodAndBad} and demonstrated in several examples in
Section~\ref{ssec:covering}.

\subsection{A local estimate}\label{ssec:local}
Let us begin with the discussion of the local estimate. To this end, we use the following result, which is a straightforward
adaptation of Example~5.11 in \cite{Martin-22}; cf.\ also Lemma~5.4 and Proposition~5.10 there.

\begin{prop}\label{prop:localEstimate}
	Let $\cQ \subset \RR^d$ be a non-empty bounded convex open set, and let $\cE \subset \cQ$ be measurable with positive measure.
	Suppose that $g \in C^\infty(\cQ) \setminus\{0\}$ satisfies
	\[
		\forall\alpha \in \NN^d
		,\quad
		\norm{\partial^\alpha g}_{L^\infty(\cQ)}
		\leq
		A B^{\alpha} (\abs{\alpha}!)^\mu
	\]
	with some $A > 0$, $B \in (0,\infty)^d$, and $\mu \in [0,1]$. Then, there is a constant $K = K(d,\mu) > 0$, only depending
	on $\mu$ and the dimension $d$, such that
	\[
		\norm{g}_{L^2(\cQ)}
		\leq
		\biggl( \frac{K\vert\cQ\vert}{\vert\cE\vert} \biggr)^{KC_{\mu}}
			\norm{g}_{L^2(\cE)}
	\]
	with
	\[
		C_{\mu}
		=
		\begin{cases}
			1 - \log\norm{g}_{L^\infty} + \log(A) + (\abs{B}_1\diam\cQ)^{\frac{1}{1-\mu}} ,& 0 \leq \mu < 1,\\[5pt]
			\bigl( 1 - \log\norm{g}_{L^\infty} + \log(A) \bigr) e^{K\abs{B}_1\diam\cQ} ,& \mu = 1.
		\end{cases}
	\]
\end{prop}

The following consequence of Proposition~\ref{prop:localEstimate} is tailored towards a combination with the covering strategy
described in Section~\ref{ssec:good} below. Here and from now on, we adopt the notation $\ell x \in \RR^d$ for the
coordinatewise product of $\ell \in (0,\infty)^d$ and $x \in \RR^d$.

\begin{cor}\label{cor:localEstimatewChangofVariables}
	Let $\cQ = (0,1)^d$ be the unit hypercube in $\RR^d$, and let $Q = \Psi(\cQ)$, where
	$\Psi \colon \RR^d \to \RR^d$ is given by $\Psi(x) = x_0 + \ell x$, $x \in \RR^d$, with some $x_0 \in \RR^d$ and
	$\ell \in (0,\infty)^d$. Moreover, suppose that $f \in C^\infty(Q)$ satisfies
	\[
		\forall\alpha \in \NN^d
		,\quad
		\norm{\partial^\alpha f}_{L^2(Q)}
		\leq
		A B^{\alpha} (\abs{\alpha}!)^\mu \norm{f}_{L^2(Q)}
	\]
	with some $A > 0$, $B \in (0,\infty)^d$, and $\mu \in [0,1]$. Then, there is a constant $K = K(d,\mu) > 0$, only depending
	on $\mu$ and the dimension $d$, such that for every measurable set $E \subset Q$ of positive measure we have
	\[
		\norm{f}_{L^2(Q)}^2
		\leq
		\biggl( \frac{K\vert Q\vert}{\vert E\vert} \biggr)^{KC_{\mu}}
			\norm{f}_{L^2(E)}^2
	\]
	with
	\[
		C_{\mu}
		=
		\begin{cases}
			1 + \log(A) + (\abs{\ell B}_1)^{\frac{1}{1-\mu}} ,& 0 \leq \mu < 1,\\
			\bigl( 1 + \log(A) \bigr) e^{K\abs{\ell B}_1} ,& \mu = 1.
		\end{cases}
	\]
\end{cor}

\begin{proof}
	There is nothing to prove if $f = 0$, so we may suppose that $f \neq 0$. Define the function $g \in C^\infty(\cQ)$ by
	\[
		g(x)
		:=
		\frac{\sqrt{\vert\cQ\vert}}{\norm{f \circ \Psi}_{L^2(\cQ)}} (f \circ \Psi)(x)
		=
		\frac{\sqrt{\vert\cQ\vert}}{\norm{f \circ \Psi}_{L^2(\cQ)}} f(x_0 + \ell x)
		,\quad
		x \in \cQ
		,
	\]
	which obviously satisfies
	\begin{equation}\label{eq:lowerBoundg}
		\norm{g}_{L^\infty(\cQ)}
		=
		\frac{\sqrt{\vert\cQ\vert}}{\norm{f \circ \Psi}_{L^2(\cQ)}} \norm{f \circ \Psi}_{L^\infty(\cQ)}
		\geq
		1
		,
	\end{equation}
	and also
	\begin{equation}\label{eq:boundDerg}
		\forall\alpha \in \NN^d
		,\quad
		\norm{\partial^\alpha g}_{L^2(\cQ)}
		\leq
		(A\sqrt{\vert\cQ\vert}) (\ell B)^\alpha (\abs{\alpha}!)^\mu
	\end{equation}
	by the hypothesis on $f$.

	Now, $\cQ$ satisfies the cone condition, so that by \cite[Theorem~4.12]{AdamsF-03} we have the Sobolev embedding
	$W^{d,2}(\cQ) \hookrightarrow L^\infty(\cQ)$, that is, there is a constant $C$, depending only on the dimension $d$, such that
	\[
		\forall u\in W^{d,2}(\cQ)
		,\quad
		\norm{u}_{L^\infty(\cQ)}
		\leq
		C\norm{u}_{W^{d,2}(\cQ)}
		.
	\]
	Taking into account that for multiindices $\alpha,\beta \in \NN^d$ with $\abs{\beta} \leq d$ we have
	\[
		\abs{\alpha+\beta}!
		\leq
		2^{\abs{\alpha}+\abs{\beta}} \abs{\alpha}! \abs{\beta}!
		\leq
		2^d d! 2^{\abs{\alpha}} \abs{\alpha}!
		,
	\]
	we then deduce from \eqref{eq:boundDerg} for every $\alpha \in \NN^d$ that
	\begin{align*}
		\norm{ \partial^\alpha g }_{L^\infty(\cQ)}^2
		&\leq
		C^2 \norm{ \partial^\alpha g }_{W^{d,2}(\cQ)}
			=
			C^2 \sum_{\abs{\beta} \leq d} \norm{ \partial^{\alpha + \beta}g }_{L^2(\cQ)}^2\\
		&\leq
		\bigl( C A \sqrt{\vert\cQ\vert} \bigr)^2 (\ell B)^{2\alpha} \sum_{\abs{\beta} \leq d} (\ell B)^{2\beta}
			(\abs{\alpha + \beta}!)^{2\mu}\\
		&\leq
		\bigl( C 2^{d\mu} (d!)^\mu A\sqrt{\vert\cQ\vert} \bigr)^2 (2^\mu \ell B)^{2\alpha} (\abs{\alpha}!)^{2\mu}
			\sum_{\abs{\beta}\leq d} (\ell B)^{2\beta}\\
		&\leq
		\bigl( C 2^{d\mu} (d!)^\mu A\sqrt{\vert\cQ\vert} \bigr)^2 (1 + \abs{\ell B}^2)^d (2^\mu \ell B)^{2\alpha}
			(\abs{\alpha}!)^{2\mu}
		,
	\end{align*}
	and, thus,
	\begin{equation}\label{eq:boundDer}
		\forall\alpha \in \NN^d
		,\quad
		\norm{ \partial^\alpha g }_{L^\infty(\cQ)}
		\leq
		\bigl( C 2^d d! A\sqrt{\vert\cQ\vert} \bigr) (1 + \abs{\ell B}^2)^{d/2} (2 \ell B)^\alpha (\abs{\alpha}!)^\mu
		.
	\end{equation}

	In light of \eqref{eq:lowerBoundg} and \eqref{eq:boundDer}, the function $g$ on $\cQ$ satisfies the hypotheses of
	Proposition~\ref{prop:localEstimate} with $A$ and $B$ replaced by $ C 2^d d! A\sqrt{\vert\cQ\vert} (1 + \abs{\ell B}^2)^{d/2}$ and
	$2 \ell B$, respectively. Applying Proposition~\ref{prop:localEstimate} with the measurable subset
	$\cE := \Psi^{-1}(E) \subset\cQ$, we therefore conclude that
	\[
		\frac{\norm{f}_{L^2(Q)}^2}{\norm{f}_{L^2(E)}^2}
		=
		\frac{\norm{g}_{L^2(\cQ)}^2}{\norm{g}_{L^2(\cE)}^2}
		\leq
		\biggl( \frac{K\vert\cQ\vert}{\vert\cE\vert} \biggr)^{KC_{\mu,\ell}}
		=
		\biggl( \frac{K\vert Q\vert}{\vert E\vert} \biggr)^{KC_{\mu,\ell}}
		,
	\]
	upon a suitable adaptation of the constant $K$, especially in order to subsume the term $\log(1 + \abs{\ell B}^2)$ into
	$(\abs{\ell B}_1)^{\frac{1}{1-\mu}}$ and $e^{K\abs{\ell B}_1}$ in $C_{\mu,\ell}$, respectively. This completes the proof.
\end{proof}

\begin{rk}
	(1)
	The restriction to the very specific type of sets $Q$ in Corollary~\ref{cor:localEstimatewChangofVariables} is due to the
	Sobolev embedding used in the proof as it is imperative that the constant $K > 0$ in the statement can be chosen uniformly
	with respect to the corresponding Sobolev constants. This is accomplished exactly via the change of variables. The specific
	restriction to the hypercube for $\cQ$, however, is not essential and could simply be extended to a larger class of sets, as
	long as a Sobolev embedding with a uniform constant remains available.

	(2)
	It is worth to note that the case $\mu < 1$ in Corollary~\ref{cor:localEstimatewChangofVariables} could alternatively also be
	handled using the more classical analytic approach by Kovrijkine (instead of Proposition~\ref{prop:localEstimate}),
	cf.\ \cite{EgidiS-21} and \cite{DickeS-22} and also Appendix~\ref{app:local} below, which would not require the use of the
	Sobolev embedding. This does, however, not seem to work for the critical case of $\mu = 1$ because then the bounds on the
	partial derivatives of $f$ are not strong enough to control the domain of convergence of a suitable Taylor expansion as is
	possible for $\mu < 1$.
\end{rk}

\subsection{Covering with good and bad elements}\label{ssec:good}

The following result formulates the general scheme to obtain global uncertainty estimates by means of the local estimate provided
by Corollary~\ref{cor:localEstimatewChangofVariables} via a covering strategy. Here, $\Gamma \subset \RR^d$ could be the whole
domain under consideration, say $\RR^d$ or a strip as in the context of Section~\ref{sec:results}, but could also be a proper
subdomain thereof, depending on the situation at hand.

\begin{lem}\label{lem:goodAndBad}
	Let $\Gamma \subset \RR^d$ be open, and let $\{Q_j\}_{j \in J}$ be a finite or countably infinite family of open sets
	$Q_j \subset \RR^d$ such that almost everywhere we have
	\[
		\bmone_\Gamma
		\leq
		\sum_{j \in J} \bmone_{Q_j}
		\leq
		\kappa\bmone_\Gamma
	\]
	with some $\kappa \geq 1$. Moreover, suppose that $f \in C^\infty(\Gamma)$ satisfies
	\[
		\forall \alpha \in \NN^d
		,\quad
		\norm{ \partial^\alpha f }_{L^2(\Gamma)}
		\leq
		D B^{\alpha} (\abs{\alpha}!)^\mu		
	\]
	with some $D > 0$, $B \in (0,\infty)^d$, and $\mu \in [0,1]$. Then, for every $\eps > 0$ there is a subset
	$J_\good = J_\good(\eps) \subset J$ of indices such that
	\[
		\norm{f}_{L^2(\Gamma)}^2
		\leq
		\sum_{j \in J_\good} \norm{f}_{L^2(Q_j)}^2 + \eps D^2
		,
	\]
	while for each $j \in J_\good$ we have
	\begin{equation}\label{eq:good}
		\forall \alpha \in \NN^d
		,\quad
		\norm{\partial^\alpha f}_{L^2(Q_j)}
		\leq
		\Bigl( \frac{2^d\kappa}{\eps} \Bigr)^{1/2} (2B)^{\alpha} (\abs{\alpha}!)^\mu \norm{f}_{L^2(Q_j)}
		.
	\end{equation}
\end{lem}

\begin{proof}
	Given $\eps > 0$, we define $J_\good \subset J$ as the subset of indices $j \in J$ for which \eqref{eq:good} is satisfied; note
	that this subset might a priori be empty. In any case, taking into account that
	\[
		\norm{f}_{L^2(\Gamma)}^2
		\leq
		\sum_{j \in J_\good} \norm{f}_{L^2(Q_j)}^2 + \sum_{j\in J\setminus J_\good} \norm{f}_{L^2(Q_j)}^2
		,
	\]
	we only have to show that
	\[
		\sum_{j \in J\setminus J_\good} \norm{f}_{L^2(Q_j)}^2
		\leq
		\eps D^2
		.
	\]
	To this end, we observe that for each $j \in J \setminus J_\good$ there is by definition some $\beta \in \NN^d$ such that
	\begin{align*}
		\norm{f}_{L^2(Q_j)}^2
		&<
		\frac{\eps}{2^d\kappa} \frac{1}{(2B)^{2\beta} (\abs{\beta}!)^{2\mu}} \norm{\partial^\beta f}_{L^2(Q_j)}^2\\
		&\leq
		\frac{\eps}{2^d\kappa} \sum_{\alpha \in \NN^d} \frac{1}{(2B)^{2\alpha} (\abs{\alpha}!)^{2\mu}}
			\norm{\partial^\alpha f}_{L^2(Q_j)}^2
		,
	\end{align*}
	and summing over $j \in J \setminus J_\good$ together with the hypotheses on $f$ and the family $\{Q_j\}_{j\in J}$ gives
	\begin{align*}
		\sum_{j \in J\setminus J_\good} \norm{f}_{L^2(Q_j)}^2
		&\leq
		\frac{\eps}{2^d} \sum_{\alpha \in \NN^d} \frac{1}{(2B)^{2\alpha} (\abs{\alpha}!)^{2\mu}}
			\norm{\partial^\alpha f}_{L^2(\Gamma)}^2\\
		&\leq
		\frac{\eps}{2^d} D^2 \sum_{\alpha \in \NN^d} \frac{1}{4^{\abs{\alpha}}}
			=
			\eps D^2
	\end{align*}
	since
	\[
		\sum_{\alpha \in \NN^d} \frac{1}{4^{\abs{\alpha}}}
		=
		\sum_{m=0}^\infty \frac{1}{4^m} \binom{m+d-1}{d-1}
		\leq
		\sum_{m=0}^\infty \frac{1}{4^m} 2^{m+d-1}
		=
		2^d
		.\qedhere
	\]
\end{proof}

\begin{rk}
	The quantity $\kappa$ in Lemma~\ref{lem:goodAndBad} can be interpreted as the maximal essential overlap between the covering
	sets. The elements $Q_j$ with $j \in J_\good$ are traditionally referred to as \emph{good}, which explains the subscript in
	$J_\good$. It is worth to note that the extreme cases for $J_\good$ are, of course, $J_\good = J$ and $J_\good = \emptyset$, in
	which case we have $\norm{f}_{L^2(\Gamma)}^2 \leq \sum_{j \in J_\good} \norm{f}_{L^2(Q_j)}^2$ and
	$\norm{f}_{L^2(\Gamma)}^2 \leq \eps D^2$, respectively. The latter case is usually not interesting as it trivially leads to a
	unique continuation estimate of the form we are looking for (with any measurable subset $\omega$). In this sense, we may always
	assume that $J_\good$ is not empty.
\end{rk}

On the good covering elements $Q_j$ in Lemma~\ref{lem:goodAndBad} (i.e.\ $j \in J_\good$), the bound provided in \eqref{eq:good}
is designed to guarantee applicability of the local estimate from Corollary~\ref{cor:localEstimatewChangofVariables} with
$A = (2^d\kappa/\eps)^{1/2}$ and with $B$ replaced by $2B$. Together with the bound from Lemma~\ref{lem:goodAndBad}, this then
leads to a global (on $\Gamma$) uncertainty estimate of the desired form, as long as the constant relating the $L^2$-norm of $f$
on $Q_j$, $j \in J$, to the one on the subset of $Q_j$ in Corollary~\ref{cor:localEstimatewChangofVariables} is uniform over
$j \in J$; it would, in fact, suffice it to be uniform over $j \in J_\good$, but $J_\good$ is usually not known explicitly. An
instance of this is summarized in the following result.

\begin{cor}\label{cor:goodAndBad}
	In the situation of Lemma~\ref{lem:goodAndBad}, suppose, in addition, that each $Q_j$ is as in
	Corollary~\ref{cor:localEstimatewChangofVariables} the affine image of the hypercube $(0,1)^d$ with corresponding scaling
	parameters $\ell_j \in (0,\infty)^d$ such that the coordinatewise supremum $\ell := \sup_{j\in J} \ell_j$ belongs to
	$(0,\infty)^d$. Then, there is a constant $K = K(d,\mu,\kappa) > 0$ such that for every $\eps > 0$ and every measurable subset
	$\omega \subset \Gamma$ with
	\[
		\vartheta
		:=
		\inf_{j\in J} \frac{\vert\omega \cap Q_j\vert}{\vert Q_j\vert}
		>
		0
		,
	\]
	we have
	\begin{equation}\label{eq:genUC}
		\norm{f}_{L^2(\Gamma)}^2
		\leq
		\biggl( \frac{K}{\vartheta} \biggr)^{KC_{\eps,\mu}} \norm{f}_{L^2(\omega)}^2 + \eps D^2
		,
	\end{equation}
	where
	\[
		C_{\eps,\mu}
		=
		\begin{cases}
			1 - \log\eps + (\abs{\ell B}_1 )^{\frac{1}{1-\mu}}, & 0 \leq \mu < 1,\\
			(1 - \log\eps) \euler^{K\abs{\ell B}_1}, & \mu = 1.
		\end{cases}
	\]
\end{cor}

\subsection{Choosing suitable coverings}\label{ssec:covering}
In order to apply Corollary~\ref{cor:goodAndBad}, the specific choices of the subdomain $\Gamma$ and a corresponding essential
covering $\{Q_j\}_{j\in J}$ highly depend on the situation at hand, especially with respect to having a good control over the
quantities $\ell \in (0,\infty)^d$ and $\vartheta > 0$. We describe here two examples with the focus on the case where the
underlying domain is
\[
	M
	=
	\RR^d \times (0,2\pi)^d
	\subset
	\RR^{2d}
	,
\]
which addresses upon a suitable identification the cylinder $\RR^d \times \TT^d$ encountered in Section~\ref{sec:results} for
the Baouendi--Grushin operators. Other cases like $\RR^d$ or $\RR^{2d}$ are completely analogous and are
briefly commented on in Remarks~\ref{rk:ucThick} and~\ref{rk:ucDecay} below.

We consider functions $f \in C^\infty(M)$ satisfying
\begin{equation}\label{eq:genSmoothing}
	\forall \alpha,\beta \in \NN^d
	,\quad
	\norm{ \partial_x^\alpha \partial_y^\beta f }_{L^2(M)}
	\leq
	D c_x^{\abs{\alpha}} c_y^{\abs{\beta}} (\alpha!)^\mu (\beta!)^\mu
\end{equation}
with some constants $D,c_x,c_y > 0$ and $\mu \in [0,1]$, which fit into the framework of Lemma~\ref{lem:goodAndBad} and
Corollary~\ref{cor:goodAndBad} with $B = (c_x\bone , c_y\bone) \in (0,\infty)^{2d}$, where we used the short-hand notation
$\bone = (1,\dots,1) \in \RR^d$.

Our first example now addresses the case of thick subsets of $M$, which applies when no additional information on the decay of $f$
in \eqref{eq:genSmoothing} is available. This situation is encountered in Theorem~\ref{thm:ucthickbg}, and in a variant on
$\RR^d$ also in Example~\ref{ex:ucfracheat} below.

\begin{ex}\label{ex:ucThick}
	Suppose that $f \in C^\infty(M)$ satisfies \eqref{eq:genSmoothing} and that $\omega \subset M$ is $(\theta,L_x,L_y)$-thick in
	$M$ in the sense of Definition~\ref{def:thickStrip}. We may then cover $\Gamma = M$, up to a set of measure zero, with a
	countable number of translates $Q_j$, $j \in J$, of $\Lambda_{L_x,L_y} = (0,L_x)^d \times (0,L_y)^d$ with essential overlap
	$\kappa \leq 2^d$. Each of these hyperrectangles $Q_j$ is the affine image of the hypercube
	$\cQ = (0,1)^d \times (0,1)^d = (0,1)^{2d}$ with corresponding scaling parameters
	$\ell = \ell_j = (L_x\bone , L_y\bone) \in (0,\infty)^{2d}$, so that
	\[
		\abs{\ell B}_1
		=
		dL_xc_x + dL_yc_y
		.
	\]
	Applying Corollary~\ref{cor:goodAndBad} and suitably adapting the constant $K = K(d,\mu)$, we conclude that \eqref{eq:genUC}
	holds with $\vartheta \geq \theta > 0$ and
	\[
		C_{\eps,\mu}
		=
		\begin{cases}
			1 - \log\eps + ( L_xc_x + L_yc_y  )^{\frac{1}{1-\mu}} ,& 0 \leq \mu < 1,\\[5pt]
			(1 - \log\eps)\exp( KL_xc_x + KL_yc_y ) ,& \mu = 1.
		\end{cases}
	\]
\end{ex}

\begin{rk}\label{rk:ucThick}
	An analogue to Example~\ref{ex:ucThick} is available if $M$ is replaced by $\RR^d$ or $\RR^{2d}$, where $\RR^{2d}$ is, in
	fact, just another instance of $\RR^d$ with $d$ replaced by $2d$. The case of $\RR^d$ differs from the above essentially only in
	the way that all terms corresponding to the $y$ coordinates can be removed; the situation is then even a bit simpler since the
	essential covering of $\RR^d$ with translates of $(0,L_x)^d$ can be achieved without any overlap, that is, $\kappa = 1$.
\end{rk}

\begin{ex}\label{ex:ucfracheat}
	As an illustration of Remark~\ref{rk:ucThick}, let us consider the fractional heat semigroup $(e^{-t(-\Delta)^s})_{t\geq0}$
	acting on $L^2(\RR^d)$, with $s > 0$ a positive real number. We prove in Lemma~\ref{lem:smoothfracheat} below that this
	semigroup enjoys for all $t > 0$, $\alpha \in \NN^d$, and $g \in L^2(\RR^d)$, the smoothing estimates
	\[
		\norm{ \partial^{\alpha}_x (\euler^{-t(-\Delta)^s}g) }_{L^2(\RR^d)}
		\leq
		\biggl( \frac d{2st} \biggr)^{\frac{\abs{\alpha}}{2s}}\,(\alpha!)^{\frac1{2s}}\,\norm{g}_{L^2(\RR^d)}
		.
	\]
	Therefore, when $s\geq1/2$, every function $f = \euler^{-t(-\Delta)^s}g$ satisfies the $\RR^d$-variant of
	\eqref{eq:genSmoothing} with $D = \norm{g}_{L^2(\RR^d)}$, $c_x = (\frac{d}{2st})^{\frac{1}{2s}}$, and
	$\mu = \frac{1}{2s} \in (0,1]$. We therefore conclude from Remark~\ref{rk:ucThick} that there exists a positive constant
	$K > 0$, depending only on $s$ and the dimension $d$, such that for every $(\theta,L)$-thick set $\omega \subset \RR^d$, with
	$\theta \in (0,1]$ and $L > 0$, and all $t > 0$, $\eps > 0$ and $g \in L^2(\RR^d)$, we have
	\begin{equation}\label{eq:ucfracheat}
		\norm{ \euler^{-t(-\Delta)^s}g }^2_{L^2(\RR^d)}
		\leq
		\biggl( \frac K{\theta} \biggr)^{KC_{\varepsilon,s,t}} \norm{ \euler^{-t(-\Delta)^s}g }^2_{L^2(\omega)} +
			\eps \norm{g}^2_{L^2(\RR^d)}
	\end{equation}
	with
	\[
		C_{\varepsilon,s,t}
		=
		\begin{cases}
			1 - \log\eps + \biggl( \dfrac{L}{t^{\frac1{2s}}} \biggr)^{\frac{2s}{2s-1}} ,& s > 1/2,\\[15pt]
			(1-\log\eps) \exp\biggl(\dfrac{KL}{t} \biggr) ,& s = 1/2.
		\end{cases}
	\]

	Let us make some bibliographical comments:

	(1)
	Non-quantitative estimates of the form 
	\[
		\norm{ \euler^{-t(-\Delta)^s}g }^2_{L^2(\RR^d)}
		\leq
		C_{\eps,s,t,\theta,L} \norm{ \euler^{-t(-\Delta)^s}g }^2_{L^2(\omega)} + \eps\norm{g}^2_{L^2(\RR^d)}
	\]
	have already been stated in \cite[Proposition~5.2]{AM}. In fact, this result states non-quantitative unique continuation
	estimates for evolution operators associated with more general Fourier multipliers of the form $F(\abs{D_x})$, where
	$F \colon [0,+\infty) \to \RR$ is a continuous function bounded from below, generating a quasi-analytic sequence $\mathcal M^F$,
	see \cite[Section~2.2]{AM} for the definition of this notion. The merit of the estimates \eqref{eq:ucfracheat} is a
	precise description of the constant $C_{\varepsilon,s,t,\theta,L}$ in the case of the fractional heat equations, which correspond
	to the functions $F(t) = t^{2s}$.

	(2)
	Similar unique continuation estimates have also been obtained for the heat equation (case $s=1$) considered in bounded domains of
	$\RR^d$, see, e.g., \cite[Section~2.1]{P}.
\end{ex}

In some situations, the function $f$ under consideration exhibits also a strong decay in $L^2$-sense with respect to the $x$
coordinates, say
\begin{equation}\label{eq:genDecay}
	\norm{ \euler^{c_0 \abs{x}^{1/\nu}}f }_{L^2(M)}
	\leq
	D
\end{equation}
with some $c_0, \nu > 0$ and the same $D > 0$ as in \eqref{eq:genSmoothing}. In this case, for every $R > 0$ and every choice of
measurable $\Gamma \subset M$ satisfying $\Gamma \supset (-R,R)^d \times (0,2\pi)^d$ we have
\[
	\norm{f}_{L^2(M\setminus\Gamma)}
	\leq
	\euler^{-c_0R^{1/\nu}} \norm{ \euler^{c_0 \abs{x}^{1/\nu}}f }_{L^2(M\setminus\Gamma)}
	\leq
	\euler^{-c_0R^{1/\nu}}D
\]
and, therefore,
\begin{equation}\label{eq:remainder}
	\norm{f}_{L^2(M)}^2
	\le
	\norm{f}_{L^2(\Gamma)}^2 + \euler^{-2c_0R^{1/\nu}} D^2
	.
\end{equation}
Here, upon choosing $R$ large enough, the term $\euler^{-2c_0R^{1/\nu}}$ can be hidden inside $\eps$ in \eqref{eq:genUC}, while
the term $\norm{f}_{L^2(\Gamma)}^2$ is to be addressed via Corollary~\ref{cor:goodAndBad}; note that
always $\norm{\partial_x^\alpha \partial_y^\beta f}_{L^2(\Gamma)} \leq \norm{\partial_x^\alpha \partial_y^\beta f}_{L^2(M)}$, so
that we can use \eqref{eq:genSmoothing} on $\Gamma$ instead of $M$ with the same constants. In the context of
Corollary~\ref{cor:goodAndBad}, this has the benefit that only the
part $\Gamma \cap \omega$ of the sensor set $\omega$ plays a role and the remainder $\omega \setminus \Gamma$ can be discarded.
The following example demonstrates this for sensor sets $\omega \subset M$ with merely positive measure, but no other additional
information. This situation is encountered in Theorem~\ref{thm:ucShubinDecay}, as well as in Example~\ref{ex:anharmonic}.

\begin{ex}\label{ex:ucDecay}
	Suppose that $f \in C^\infty(M)$ satisfies \eqref{eq:genSmoothing} and \eqref{eq:genDecay} and that $\omega \subset M$ has
	positive measure. We choose $R > 0$ such that the intersection of $\omega$ with the hyperrectangle
	$\Gamma = (-R,R)^d \times (0,2\pi)^d$ has positive measure, that is,
	\[
		\theta_R
		:=
		\frac{\vert\omega \cap \Gamma\vert}{\vert\Gamma\vert}
		>
		0
		.
	\]
	Here, $\Gamma$ itself is the affine image of the hypercube $(0,1)^{2d}$ with corresponding scaling parameters
	$\ell = (R\bone , 2\pi\bone)$, so that
	\[
		\abs{\ell B}_1
		=
		dRc_x + 2\pi dc_y
		.
	\]
	We may then cover $\Gamma$ with just the one element family $\{\Gamma\}$ with $\kappa = 1$ and apply
	Corollary~\ref{cor:goodAndBad} with the particular choice $\eps = \euler^{-2c_0R^{1/\nu}}$. Upon suitably adapting the constant
	$K = K(d,\mu)$, and taking into account \eqref{eq:remainder}, this gives
	\[
		\norm{f}_{L^2(M)}^2
		\leq
		\Bigl( \frac{K}{\theta_R} \Bigr)^{KC_{\eps,\mu}} + 2\euler^{-2c_0R^{1/\nu}}D^2
	\]
	with
	\[
		C_{\eps,\mu}
		=
		\begin{cases}
			1 + c_0R^{1/\nu} + ( Rc_x + c_y )^{\frac{1}{1-\mu}} ,& 0 \leq \mu < 1,\\
			(1 + c_0R^{1/\nu})\exp( KRc_x + Kc_y ) ,& \mu = 1.
		\end{cases}
	\]
\end{ex}

\begin{rk}\label{rk:ucDecay}
	(1)
	In the situation of \eqref{eq:genSmoothing} and \eqref{eq:genDecay}, we could also consider sensor sets that have not only
	positive measure but are even thick in $M$ with a decaying density and with respect to a variable local scale in the $x$
	coordinates, cf.\ \cite{AlphonseS, DickeSV, DickeSV-23, DickeSV-23b}. The strategy of proof would combine the above scheme with the
	well-known Besicovitch covering theorem, just as in the mentioned references. We decided not to pursue this here.

	(2)
	An analogue to Example~\ref{ex:ucDecay} is available if $M$ is replaced by $\RR^d$. The difference is only that just as in
	Remark~\ref{rk:ucThick} all terms in Example~\ref{ex:ucDecay} corresponding to the $y$ coordinates can be removed.
\end{rk}

\section{Smoothing properties of the fractional Baouendi--Grushin equations}\label{sec:Grushin}

This section is devoted to studying the smoothing properties of the evolution equations associated by the fractional
Baouendi--Grushin operator $(-\Delta_{\gamma})^{(1+\gamma)/2}$. This is a crucial step for obtaining the unique continuation
estimates stated in Theorem~\ref{thm:ucthickbg} via the framework from Section~\ref{sec:uniqueCont}. We
also discuss the case of eigenfunctions of the operator $\Delta_{\gamma}$, or finite linear combinations thereof, which
additionally enjoy localization properties.

\subsection{Prolegomena}\label{subsec:pro}
First of all, we introduce some notations, and recall some very well-known facts of spectral analysis.

\subsubsection{Diagonalization}
After diagonalizing the Laplace operator $\Delta_y$ on the torus $\mathbb T^d$, the Baouendi--Grushin operator $\Delta_{\gamma}$
is transformed as
\begin{equation}\label{eq:diagfourier}
	\Delta_x + \vert x\vert^{2\gamma}\Delta_y\,\rightsquigarrow\,\Delta_x - \vert n\vert^2\vert x\vert^{2\gamma}.
\end{equation}
This motivates to introduce, for each frequency $n\in\mathbb Z^d$, the anharmonic oscillator $H_{\gamma,n}$ with variably scaled
potential, defined by
\[
	H_{\gamma,n} = -\Delta_x + \vert n\vert^2\vert x\vert^{2\gamma},\quad x\in\mathbb R^d.
\]
Recall that when $n\ne0$, each of these operators, equipped with the respective domain
\[
	D(H_{\gamma,n}) = \big\{g\in L^2(\mathbb R^d) : H_{\gamma,n}g\in L^2(\mathbb R^d)\big\},
\]
has purely discrete spectrum consisting of a sequence of positive eigenvalues diverging to $+\infty$. Moreover, the negatives of
each $H_{\gamma,n}$ and of its fractional powers $(H_{\gamma,n})^s$ with $s>0$ (understood via the standard functional calculus)
generate strongly continuous contraction semigroups on $L^2(\mathbb R^d)$. When $n = 0$, the operator $H_{\gamma,0}$ reduces to
the negative Laplacian $-\Delta_x$, and $-H_{\gamma,0}$ likewise generates a strongly continuous contraction semigroups on
$L^2(\mathbb R^d)$ (as do the operators $-(H_{\gamma,0})^s$). Consequently, we have for all $t\geq0$, $s>0$ and $g\in L^2(\var)$,
\begin{equation}\label{eq:diag}
	e^{-t(-\Delta_{\gamma})^s}g = \sum_{n\in\spt}(e^{-t(H_{\gamma,n})^s}g_n)\otimes\varphi_n,
\end{equation}
with the partial Fourier coefficients
\begin{equation}\label{eq:fouriercoeff}
	g_n = \int_M\overline{\varphi_n(y)}g(\cdot,y)\,\mathrm dy\quad\text{where}\quad\varphi_n(y) = e^{in\cdot y}.
\end{equation}
It also follows from this discussion that the spectrum of the operator $-\Delta_{\gamma}$ is equal to $[0,+\infty)$ (due to the contribution of the frequency $n=0$ in \eqref{eq:diagfourier}), it coincides with its essential spectrum, and it also contains embedded eigenvalues that will be described in Section \ref{subsubsec:eigenfunctions} below.

\subsubsection{Unitary transform}
For every non-zero frequency $n\in\mathbb Z^d\setminus\{0\}$, let us introduce the isometry $M_{\gamma,n}$ on $L^2(\mathbb R^d)$
by
\begin{equation}\label{eq:isometry}
	M_{\gamma,n}g = \vert n\vert^{\frac d{2(1+\gamma)}}g(\vert n\vert^{\frac1{1+\gamma}}\cdot)
	,\quad
	g\in L^2(\mathbb R^d).
\end{equation}
A straightforward computation shows that
\begin{equation}\label{eq:sim}
	(M_{\gamma,n})^*(H_{\gamma,n})^sM_{\gamma,n} = \vert n\vert^{\frac{2s}{1+\gamma}}(H_{\gamma})^s,
\end{equation}
where $H_{\gamma}$ denotes the standard anharmonic oscillator 
\[
	H_{\gamma} = -\Delta_x + \vert x\vert^{2\gamma},\quad x\in\mathbb R^d.
\]
In the following, we denote by $(\lambda_{\gamma,m})_m$ the sequence of eigenvalues of the maximal realization of the operator
$H_{\gamma}$ on $L^2(\mathbb R^d)$ in non-decreasing order and counting multiplicities, and by $(\phi_{\gamma,m})_m$ with
$\phi_{\gamma,m}\in L^2(\mathbb R^d)$ a corresponding orthonormal basis of eigenfunctions for $H_\gamma$ such that each
$\phi_{\gamma,m}$ is associated to $\lambda_{\gamma,m}>0$.

\subsubsection{Eigenfunctions}\label{subsubsec:eigenfunctions}
In view of the similarity relation \eqref{eq:sim}, it is clear that the eigenvalues of the Baouendi--Grushin operator
$-\Delta_{\gamma}$ are precisely given by
\[
	\lambda_{\gamma,n,m}
	=
	\vert n\vert^{\frac2{1+\gamma}}\lambda_{\gamma,m}
	,\quad
	n\in\spt\setminus\{0\},\,m \in \NN
	,
\]
and the function
\begin{equation}\label{eq:eigenfunctions}
	\psi_{\gamma,n,m} = (M_{\gamma,n}\phi_{\gamma,m})\otimes\varphi_n,
\end{equation}
in $L^2(\var)$ is an associated normalized eigenfunction of the operator $-\Delta_{\gamma}$. Moreover, the subspace
$E_\lambda(-\Delta_\gamma)$ defined in \eqref{eq:combilineigen} can be rewritten as
\[
	E_{\lambda}(-\Delta_{\gamma})
	=
	\spane\bigl\{ \psi_{\gamma,n,m} \colon n\in\spt\setminus\{0\},\,m\in\NN,\ \lambda_{\gamma,n,m} \leq \lambda \bigr\}
	.
\]

\subsection{Smoothing properties of the evolution operators}\label{subsec:smoothingbg}

Let us now focus on the smoothing properties of the semigroups generated by Baouendi--Grushin operators. We are first interested
in the smoothing properties with respect to the variable $y\in\mathbb T^d$.

\begin{lem}\label{lem:regy}
	For all $s>0$, $t>0$ and $g\in L^2(\var)$, we have
	\begin{equation}\label{eq:regy}
		\big\Vert\partial^{\alpha}_y(e^{-t(-\Delta_{\gamma})^s}g)\big\Vert_{L^2(\var)}
		\le\bigg(\frac{1+\gamma}{2s\lambda^s_{\gamma}t}\bigg)^{\frac{\vert\alpha\vert(1+\gamma)}{2s}}\,(\alpha!)^{\frac{1+\gamma}{2s}}\,\Vert g\Vert_{L^2(\var)},
	\end{equation}
where $\lambda_{\gamma}>0$ stands for the smallest eigenvalue of the anharmonic oscillator $H_{\gamma}$.
\end{lem}

\begin{rk}
	It is worthing noticing from the above result that the evolution operators generated by the fractional Baouendi--Grushin
	operator $(-\Delta_{\gamma})^s$ enjoy (at least) analytic smoothing properties in the variable $y$ whenever $s\geq(1+\gamma)/2$.
	This point is clear from the estimates \eqref{eq:regyexp} since the analytic smoothness is equivalent to the exponential
	decrease of the Fourier transform. As a consequence, we are able to prove unique continuation estimates for the functions
	$e^{-t(-\Delta_{\gamma})^s}g$ only in this regime, and in fact, we will focus only on the case $s=(1+\gamma)/2$.
\end{rk}

\begin{proof}[Proof of Lemma \ref{lem:regy}]
	Let us begin by proving that for all $t\geq0$ and $g\in L^2(\var)$,
	\begin{equation}\label{eq:regyexp}
		\big\Vert e^{\lambda^s_{\gamma}t\vert D_y\vert^{\frac{2s}{1+\gamma}}}(e^{-t(-\Delta_{\gamma})^s}g)\big\Vert_{L^2(\var)}
		\le \Vert g\Vert_{L^2(\var)}.
	\end{equation}
	Let $t\geq0$ and $g\in L^2(\var)$ be fixed. First, we get from the diagonalization formula \eqref{eq:diag} and Parseval's
	theorem that
	\[
		\big\Vert e^{\lambda^s_{\gamma}t\vert D_y\vert^{\frac{2s}{1+\gamma}}}(e^{-t(-\Delta_{\gamma})^s}g)\big\Vert^2_{L^2(\var)} 
		= \sum_{n\in\spt}e^{2\lambda^s_{\gamma}t\vert n\vert^{\frac{2s}{1+\gamma}}}\big\Vert e^{-t(H_{\gamma,n})^s}g_n\big\Vert^2_{L^2(\mathbb R^d)},
	\]
	where $g_n$ denotes the partial Fourier coefficient defined in \eqref{eq:fouriercoeff}. Let $n\in\spt$ be fixed. On the one hand,
	when $n = 0$, the operator $H_{\gamma,n}$ reduces to the negative Laplacian $-\Delta_x$ on $\mathbb R^d$. Since the evolution
	operators $e^{-t(-\Delta_x)^s}$ are contractions on $L^2(\mathbb R^d)$, we deduce that 
	\[
		\big\Vert e^{-t(H_{\gamma,0})^s}g_0\big\Vert^2_{L^2(\mathbb R^d)}
		= \big\Vert e^{-t(-\Delta_x)^s}g_0\big\Vert^2_{L^2(\mathbb R^d)} \le \Vert g_0\Vert^2_{L^2(\mathbb R^d)}.
	\]
	On the other hand, when $n\in\mathbb Z^d\setminus\{0\}$, we deduce from the relation \eqref{eq:sim} that
	\[
		\big\Vert e^{-t(H_{\gamma,n})^s}g_n\big\Vert_{L^2(\mathbb R^d)} 
		= \big\Vert e^{-\vert n\vert^{\frac{2s}{1+\gamma}}tH_{\gamma}^s}(M_{\gamma,n})^*g_n\big\Vert_{L^2(\mathbb R^d)}.
	\]
	Moreover, we get from Plancherel's theorem that
	\begin{equation}\label{eq:spegap}
		\big\Vert e^{-tH^s_{\gamma}}\big\Vert_{\mathcal L(L^2(\mathbb R^d))}\le e^{-\lambda^s_{\gamma}t},
	\end{equation}
	which implies
	\[
		e^{2\lambda^s_{\gamma}t\vert n\vert^{\frac{2s}{1+\gamma}}}\big\Vert e^{-t(H_{\gamma,n})^s}g_n\big\Vert^2_{L^2(\mathbb R^d)}
		\le e^{2\lambda^s_{\gamma}t\vert n\vert^{\frac{2s}{1+\gamma}}}e^{-2\lambda^s_{\gamma}t\vert n\vert^{\frac{2s}{1+\gamma}}}\Vert g_n\Vert^2_{L^2(\mathbb R^d)} 
		= \Vert g_n\Vert^2_{L^2(\mathbb R^d)}.
	\]
	In a nutshell, we obtained that for all $t\geq0$ and $g\in L^2(\var)$,
	\[
		\big\Vert e^{\lambda^s_{\gamma}t\vert D_y\vert^{\frac{2s}{1+\gamma}}}(e^{-t(-\Delta_{\gamma})^s}g)\big\Vert^2_{L^2(\var)}
		\le\sum_{n = 0}^{+\infty}\Vert g_n\Vert^2_{L^2(\mathbb R^d)} 
		= \Vert g\Vert^2_{L^2(\var)},
	\]
	which proves estimate \eqref{eq:regyexp}.

	Let us now explain how to derive the estimate \eqref{eq:regy} from \eqref{eq:regyexp}. Notice first from the decomposition
	\eqref{eq:diag} and Parseval's theorem that for all $\alpha\in\mathbb N^d$ , $t > 0$, and $g\in L^2(\var)$,
	\begin{equation}\label{eq:boundpartialder}
		\begin{aligned}
			\Vert\partial^{\alpha}_y(e^{-t(-\Delta_{\gamma})^s}g)\Vert^2_{L^2(\var)}
			&=
			\sum_{n\in\spt}\vert n^{\alpha}\vert^2\Vert e^{-t(H_{\gamma,n})^s}g_n\Vert^2_{L^2(\mathbb R^d)}\\
			&\le
			\sum_{n\in\spt}\vert n\vert^{2\vert\alpha\vert}\Vert e^{-t(H_{\gamma,n})^s}g_n\Vert^2_{L^2(\mathbb R^d)}\\
			&=
			\Vert\vert D_y\vert^{\vert\alpha\vert}(e^{-t(-\Delta_{\gamma})^s}g)\Vert^2_{L^2(\var)}
			.
		\end{aligned}
	\end{equation}
	Since for all $p,q>0$, $c>0$ and $x\geq0$, we have the elementary inequality
	\begin{equation}\label{eq:keytool}
		x^pe^{-cx^q}\le\bigg(\frac p{ecq}\bigg)^{\frac pq},
	\end{equation}
	which follows from a straightforward study of function, and taking into account the estimate
	$\vert\alpha\vert^{\vert\alpha\vert}\le e^{\vert\alpha\vert}\vert\alpha\vert!$, we deduce from \eqref{eq:boundpartialder} for
	all $\alpha\in\mathbb N^d$, $t>0$ and $g\in L^2(\var)$,
	\begin{align*}
		\big\Vert\partial^{\alpha}_y(e^{-t(-\Delta_{\gamma})^s}g)\big\Vert_{L^2(\var)}
		& \le\big\Vert\vert D_y\vert^{\vert\alpha\vert}e^{-\lambda^s_{\gamma}t\vert D_y\vert^{\frac{2s}{1+\gamma}}}e^{\lambda^s_{\gamma}t\vert D_y\vert^{\frac{2s}{1+\gamma}}}(e^{-t(-\Delta_{\gamma})^s}g)\big\Vert_{L^2(\var)} \\[5pt]
		& \le\bigg(\frac{\vert\alpha\vert(1+\gamma)}{2es\lambda^s_{\gamma}t}\bigg)^{\frac{\vert\alpha\vert(1+\gamma)}{2s}}\,\Vert g\Vert_{L^2(\var)} \\[5pt]
		& \le\bigg(\frac{1+\gamma}{2s\lambda^s_{\gamma}t}\bigg)^{\frac{\vert\alpha\vert(1+\gamma)}{2s}}\,(\alpha!)^{\frac{1+\gamma}{2s}}\,\Vert g\Vert_{L^2(\var)}.
	\end{align*}
	This ends the proof of Lemma \ref{lem:regy}.
\end{proof}

Before treating the case of the variable $x\in\mathbb R^d$, let us recall and give a proof of the smoothing properties of the
fractional heat semigroups acting on the space $L^2(\mathbb R^d)$.

\begin{lem}\label{lem:smoothfracheat}
	For all $s>0$, $t>0$, $\alpha\in\mathbb N^d$ and $g\in L^2(\mathbb R^d)$, we have
	\[
		\big\Vert\partial^{\alpha}_x(e^{-t(-\Delta)^s}g)\big\Vert_{L^2(\mathbb R^d)}
		\le\bigg(\frac d{2st}\bigg)^{\frac{\vert\alpha\vert}{2s}}\,(\alpha!)^{\frac1{2s}}\,\Vert g\Vert_{L^2(\mathbb R^d)}.
	\]
\end{lem}

\begin{proof}
	Let $t>0$, $\alpha\in\mathbb N^d$ and $g\in L^2(\mathbb R^d)$ be fixed. We first deduce from Plancherel's theorem that
	\[
		\big\Vert\partial^{\alpha}_x(e^{-t(-\Delta)^s}g)\big\Vert_{L^2(\mathbb R^d)} 
		= \frac1{(2\pi)^{d/2}}\big\Vert(i\xi)^{\alpha}e^{-t\vert\xi\vert^{2s}}\widehat g\big\Vert_{L^2(\mathbb R^d)}
		\le\frac1{(2\pi)^{d/2}}\big\Vert\vert\xi\vert^{\vert\alpha\vert}e^{-t\vert\xi\vert^{2s}}\widehat g\big\Vert_{L^2(\mathbb R^d)}.
	\]
	Moreover, we deduce from \eqref{eq:keytool} and the estimates $\vert\alpha\vert^{\vert\alpha\vert}\le e^{\vert\alpha\vert}\vert\alpha\vert!$ and $\vert\alpha\vert!\le d^{\vert\alpha\vert}\alpha!$ that for all $\xi\in\mathbb R^d$,
	\[
		\vert\xi\vert^{\vert\alpha\vert}e^{-t\vert\xi\vert^{2s}}
		\le\bigg(\frac{\vert\alpha\vert}{2est}\bigg)^{\frac{\vert\alpha\vert}{2s}}
		\le\bigg(\frac d{2st}\bigg)^{\frac{\vert\alpha\vert}{2s}}(\alpha!)^{\frac1{2s}}.
	\]
	This ends the proof of Lemma \ref{lem:smoothfracheat} by using anew Plancherel's theorem.
\end{proof}

We are now in position to derive smoothing estimates in the variable $x\in\mathbb R^d$ for the fractional Baouendi--Grushin
evolution equation in the critical regime $s=(1+\gamma)/2$.

\begin{lem}\label{lem:regxbg}
	There exist some positive constants $c>0$ and $t_0\in(0,1)$ such that for all $t\in(0,t_0)$ and $g\in L^2(\var)$,
	\begin{equation}\label{eq:regxbg}
		\big\Vert\partial^{\alpha}_x(e^{-t(-\Delta_{\gamma})^{\frac{1+\gamma}2}}g)\big\Vert_{L^2(\var)}
		\le \frac{c^{1+\vert\alpha\vert}}{t^{\frac{\gamma\vert\alpha\vert}{1+\gamma}+\frac d\gamma}}\,\alpha!\,\Vert g\Vert_{L^2(\var)}.
	\end{equation}
	In the particular case of $\gamma = 1$, the same holds even for all positive times $t>0$ and the term $t^{d/\gamma}$ on the
	right-hand side of \eqref{eq:regxbg} can be skipped.
\end{lem}

\begin{proof}
	Let us first treat the case $\gamma\geq 2$. In order to alleviate the writing, we denote by $s_{\gamma} = (1+\gamma)/2$ the
	dissipation index. Let $t>0$ and $g\in L^2(\var)$ be fixed. We first deduce from Plancherel's theorem and the diagonalization
	formula \eqref{eq:diag} that 
	\begin{equation}\label{eq:plancherelbg}
		\big\Vert\partial^{\alpha}_x(e^{-t(-\Delta_{\gamma})^{s_\gamma}}g)\big\Vert^2_{L^2(\var)} 
		= \sum_{n\in\spt}\big\Vert\partial^{\alpha}_x(e^{-t(H_{\gamma,n})^{s_{\gamma}}}g_n)\big\Vert^2_{L^2(\mathbb R^d)},
	\end{equation}
	where $g_n$ is still defined as in \eqref{eq:fouriercoeff}. We now need to estimate each term in the above sum. When $n=0$, we
	deduce from Lemma \ref{lem:smoothfracheat} that there exists a positive constant $c_0>0$ only depending on $s$ and the dimension
	$d$ such that
	\begin{equation}\label{eq:fracheatsmoothing}
		\big\Vert\partial^{\alpha}_x(e^{-t(H_{\gamma,0})^{s_{\gamma}}}g_0)\big\Vert_{L^2(\mathbb R^d)} 
		= \big\Vert\partial^{\alpha}_x(e^{-t(-\Delta_x)^{s_{\gamma}}}g_0)\big\Vert_{L^2(\mathbb R^d)}
		\le\frac{c_0^{\vert\alpha\vert}}{t^{\frac{\vert\alpha\vert}{1+\gamma}}}\,(\alpha!)^{\frac1{1+\gamma}}\,\Vert g_0\Vert_{L^2(\mathbb R^d)}.
	\end{equation}
	Until the end of the proof, we consider $n\in\mathbb Z^d\setminus\{0\}$. We deduce from the relation \eqref{eq:sim} that 
	\begin{align*}
		\big\Vert\partial^{\alpha}_x(e^{-t(H_{\gamma,n})^{s_{\gamma}}}g_n)\big\Vert_{L^2(\mathbb R^d)}
		& = \big\Vert\partial^{\alpha}_x(M_{\gamma,n}e^{-\vert n\vert tH_{\gamma}^{s_{\gamma}}}(M_{\gamma,n})^*g_n)\big\Vert_{L^2(\mathbb R^d)} \\[5pt]
		& = \vert n\vert^{\frac{\vert\alpha\vert}{1+\gamma}}\big\Vert\partial^{\alpha}_x(e^{-\vert n\vert tH_{\gamma}^{s_{\gamma}}}(M_{\gamma,n})^*g_n)\big\Vert_{L^2(\mathbb R^d)}.
	\end{align*}
	In order to control the above term, we use Corollary~2.2 in \cite{Alp20b}, which states that there exist positive constants
	$c_1>1$ and $0<t_0<1$ such that for all $0<t\le t_0$, $\alpha\in\mathbb N^d$ and $g\in L^2(\mathbb R^d)$,
	\begin{equation}\label{eq:smoothfracharmoshort}
		\big\Vert\partial^{\alpha}_x(e^{-tH^{s_{\gamma}}_{\gamma}}g)\big\Vert_{L^2(\mathbb R^d)}
		\le\frac{c_1^{\vert\alpha\vert}}{t^{\frac{\gamma\vert\alpha\vert}{1+\gamma}+\frac d{\gamma}}}\,(\alpha!)^{\frac{\gamma}{\gamma+1}}\,\Vert g\Vert_{L^2(\mathbb R^d)}.
	\end{equation}
	By using \eqref{eq:spegap} anew, we also deduce that there exists another positive constant $c_2>1$ such that for all $t>t_0$,
	$\alpha\in\mathbb N^d$ and $g\in L^2(\mathbb R^d)$,
	\begin{equation}\label{eq:smoothfracharmolong}
		\big\Vert\partial^{\alpha}_x(e^{-tH^{s_{\gamma}}_{\gamma}}g)\big\Vert_{L^2(\mathbb R^d)}
		\le c_2^{1+\vert\alpha\vert}\,e^{-\lambda_{\gamma}^{s_{\gamma}}t}\,(\alpha!)^{\frac{\gamma}{\gamma+1}}\,\Vert g\Vert_{L^2(\mathbb R^d)}.
	\end{equation}
	There are two cases to consider.\\
	\textit{$\triangleright$ Case 1: $\vert n\vert t\le t_0$.} In this situation, we deduce from the estimate
	\eqref{eq:smoothfracharmoshort} that 
	\[
		\vert n\vert^{\frac{\vert\alpha\vert}{1+\gamma}}\big\Vert\partial^{\alpha}_x(e^{-\abs{n}tH_{\gamma}^{s_{\gamma}}}(M_{\gamma,n})^*g_n)\big\Vert_{L^2(\mathbb R^d)}
		\le\frac{c_1^{\vert\alpha\vert}\vert n\vert^{\frac{\vert\alpha\vert}{1+\gamma}}}{(\vert n\vert t)^{\frac{\gamma\vert\alpha\vert}{1+\gamma}+\frac{d}{\gamma}}}\,(\alpha!)^{\frac{\gamma}{\gamma+1}}\,\Vert g_n\Vert_{L^2(\mathbb R^d)}.
	\]
	Noticing that $\gamma-1\geq0$ and using the fact that $\vert n\vert\geq1$, we get that
	\begin{equation}\label{eq:smoothingfracshort}
		\vert n\vert^{\frac{\vert\alpha\vert}{1+\gamma}}\big\Vert\partial^{\alpha}_x(e^{-\vert n\vert tH_{\gamma}^{s_{\gamma}}}(M_{\gamma,n})^*g_n)\big\Vert_{L^2(\mathbb R^d)}
		\le\frac{c_1^{\vert\alpha\vert}}{t^{\frac{\gamma\vert\alpha\vert}{1+\gamma}+\frac{d}{\gamma}}}\,(\alpha!)^{\frac{\gamma}{\gamma+1}}\,\Vert g_n\Vert_{L^2(\mathbb R^d)}.
	\end{equation}
	\textit{$\triangleright$ Case 2: $\vert n\vert t>t_0$.} In this case, we use \eqref{eq:smoothfracharmolong} to obtain that 
	\[
		\vert n\vert^{\frac{\vert\alpha\vert}{1+\gamma}}\big\Vert\partial^{\alpha}_x(e^{-\vert n\vert tH_{\gamma}^{s_{\gamma}}}(M_{\gamma,n})^*g_n)\big\Vert_{L^2(\mathbb R^d)}
		\le c_2^{1+\vert\alpha\vert}\, \vert n\vert^{\frac{\vert\alpha\vert}{1+\gamma}}e^{-\lambda_{\gamma}^{s_{\gamma}}\vert n\vert t}\,(\alpha!)^{\frac{\gamma}{\gamma+1}}\,\Vert g_n\Vert_{L^2(\mathbb R^d)}.
	\]
	Combining \eqref{eq:keytool} and the estimates $\vert\alpha\vert^{\vert\alpha\vert}\le e^{\vert\alpha\vert}\vert\alpha\vert!$
	and $\vert\alpha\vert!\le d^{\vert\alpha\vert}\alpha!$ implies that 
	\[
		\vert n\vert^{\frac{\vert\alpha\vert}{1+\gamma}}e^{-\lambda_{\gamma}^{s_{\gamma}}\vert n\vert t}
		\le\bigg(\frac{\vert\alpha\vert}{e\lambda^{s_{\gamma}}_{\gamma}(1+\gamma)t}\bigg)^{\frac{\vert\alpha\vert}{1+\gamma}}
		\le\frac{c_3^{\vert\alpha\vert}}{t^{\frac{\vert\alpha\vert}{1+\gamma}}}\,(\alpha!)^{\frac1{1+\gamma}}.
	\]
	In a nutshell, we proved that
	\begin{equation}\label{eq:smoothingfraclong}
		\vert n\vert^{\frac{\vert\alpha\vert}{1+\gamma}}\big\Vert\partial^{\alpha}_x(e^{-\vert n\vert tH_{\gamma}^{s_{\gamma}}}(M_{\gamma,n})^*g_n)\big\Vert_{L^2(\mathbb R^d)}
		\le\frac{(c_2c_3)^{1+\vert\alpha\vert}}{t^{\frac{\vert\alpha\vert}{1+\gamma}}}\,\alpha!\,\Vert g_n\Vert_{L^2(\mathbb R^d)}.
	\end{equation}
	\textit{$\triangleright$ Summary:} Gathering \eqref{eq:plancherelbg} and the estimates \eqref{eq:fracheatsmoothing},
	\eqref{eq:smoothingfracshort} and \eqref{eq:smoothingfraclong}, and using the fact that $0<t<t_0<1$, we get that
	\[
		\big\Vert\partial^{\alpha}_x(e^{-t(-\Delta_{\gamma})^{s_{\gamma}}}g)\big\Vert_{L^2(\var)}
		\le\frac{(c_1c_2c_3)^{1+\vert\alpha\vert}}{t^{\frac{\gamma\vert\alpha\vert}{1+\gamma}+\frac d\gamma}}\,\alpha!\,\Vert g\Vert_{L^2(\var)}.
	\]
	The proof is ended when $\gamma\geq2$. In the case $\gamma=1$, instead of \eqref{eq:smoothfracharmoshort}, we use the estimates
	\[
		\big\Vert\partial^{\alpha}_x(e^{-tH_1}g)\big\Vert_{L^2(\mathbb R^d)}
		\le
		\frac{c_1^{\vert\alpha\vert}}{t^{\frac{\vert\alpha\vert}2}}\,\sqrt{\alpha!}\,\Vert g\Vert_{L^2(\mathbb R^d)},
	\]
	taken from Theorem~2.3 in \cite{Alp20b}, which hold for all $0<t\le t_0$ and differ from \eqref{eq:smoothfracharmoshort} by a lack of
	the factor $t^{-d/\gamma}$. The rest of the proof is then analogous to the one for $\gamma \geq 2$, except that there is no need to make the restriction $0<t<1$ at the end since the power over $t$ in the right-hand side of the estimates are the same, being equal to $\vert\alpha\vert/2$.
\end{proof}

As a consequence, we obtain the following global smoothing estimates on $\var$.

\begin{cor}\label{cor:smoothingBG}
	There exist constants $c>0$ and $t_0\in(0,1)$, depending only on $\gamma$ and the dimension $d$, such that for all
	$t_x \in (0,t_0)$, $t_y > 0$, $\alpha,\beta \in \NN^d$, and $g\in L^2(\var)$,
	\begin{equation}\label{eq:smoothingBG}
		\norm{ \partial_x^\alpha\partial_y^\beta (\euler^{-(t_x+t_y)(-\Delta_\gamma)^{\frac{1+\gamma}2}}g) }_{L^2(\var)}
		\leq
		\frac{c^{1+\abs{\alpha}+\abs{\beta}}}{(t_x)^{\frac{\gamma\vert\alpha\vert}{\gamma+1}+\frac{d}{\gamma}} (t_y)^{\abs{\beta}}}\,
			\alpha!\,\beta!\,\norm{g}_{L^2(\var)}
		.
	\end{equation}
	If $\gamma = 1$, then \eqref{eq:smoothingBG} holds even for all $t_x > 0$ and the term $t_x^{-\frac{d}{\gamma}}$ on the
	right-hand side of \eqref{eq:smoothingBG} can be skipped.
\end{cor}

\begin{proof}
	Observe that the operators $\partial_y^\beta$ and $\Delta_\gamma$ commute, and the same property holds for the operators
	$\partial_y^\beta$ and $\euler^{-t(-\Delta_\gamma)^{(1+\gamma)/2}}$ for all $t > 0$. Consequently, for all
	$(\alpha,\beta) \in \NN^{2d}$ and $t_x,t_y \geq 0$, we have
	\[
		\partial_x^\alpha\partial_y^\beta \euler^{-(t_x+t_y)(-\Delta_\gamma)^{\frac{1+\gamma}2}}
		=
		(\partial_x^\alpha \euler^{-t_x(-\Delta_\gamma)^{\frac{1+\gamma}2}}) (\partial_y^\beta \euler^{-t_y(-\Delta_\gamma)^{\frac{1+\gamma}2}})
		.
	\]
	The claim then follows from Lemmas~\ref{lem:regy} and~\ref{lem:regxbg}.
\end{proof}

Let us now discuss the optimality of the result stated in Lemma~\ref{lem:regxbg} in the particular case $\gamma=1$. Notice from
Lemma~\ref{lem:expoform} that the estimates \eqref{eq:regxbg} can be rewritten for all $t\in[0,1]$ and $g\in L^2(\var)$ as
\begin{equation}\label{eq:expoformregxbg}
	\big\Vert e^{c_1\sqrt{t}\vert D_x\vert}(e^{t\Delta_1}g)\big\Vert_{L^2(\var)}\le c_2\Vert g\Vert_{L^2(\var)}.
\end{equation}
We prove that this exponential decrease cannot be improved.

\begin{lem}\label{lem:optiregx}
	Let $\alpha>1$ and $t>0$. Then, there are no positive constants $c,c_1>0$ such that for all $g\in L^2(\var)$,
	\begin{equation}\label{eq:smoothing}
		\big\Vert e^{c\vert D_x\vert^{\para}}(e^{t\Delta_1}g)\big\Vert_{L^2(\var)}
		\le c_1\Vert g\Vert_{L^2(\var)}.
	\end{equation}
\end{lem}

\begin{proof}
	Assume to the contrary that there exist positive constants $c,c_1>0$ such that the estimate \eqref{eq:smoothing} holds. For
	each $n\in\mathbb Z^d\setminus\{0\}$, let us consider the function $\psi_n := \psi_{1,n,0}$ defined in \eqref{eq:eigenfunctions},
	which is a normalized eigenfunction of the operator $-\Delta_1$ associated with the eigenvalue $d\vert n\vert$ (since the first
	eigenvalue of the harmonic oscillator acting on $L^2(\mathbb R^d)$ is $d$). Recall that the function $\psi_n$ is given by
	\[
		\psi_n(x,y) = \bigg(\frac{\vert n\vert}{\pi}\bigg)^{d/4}e^{in\cdot y}e^{-\vert n\vert\vert x\vert^2/2},\quad (x,y)\in\var.
	\]
	Let us notice from Plancherel's theorem that the left-hand side of the inequality \eqref{eq:smoothing} applied with the function
	$\psi_n$ reads as
	\begin{align*}
		\big\Vert e^{c\vert D_x\vert^{\alpha}}(e^{t\Delta_1}\psi_n)\big\Vert^2_{L^2(\var)} 
		& = \bigg(\frac{\vert n\vert}{\pi}\bigg)^{d/2}e^{-2d\vert n\vert t}\Vert\varphi_n\Vert^2_{L^2(\mathbb T^d)}\big\Vert e^{c\vert D_x\vert^{\para}}(e^{-\vert n\vert\vert x\vert^2/2})\big\Vert^2_{L^2(\mathbb R^d)} \\[5pt]
		& = \frac{e^{-2d\vert n\vert t}}{\pi^{d/2}}\int_{\mathbb R^d}e^{c\vert\xi\vert^{\para}}e^{-\vert\xi\vert^2/(2\vert n\vert)}\,\mathrm d\xi.
	\end{align*}
	Let us check that the above integral diverges to $+\infty$ as $n$ goes to $+\infty$. Since the function $\psi_n$ is normalized,
	this then contradicts the inequality \eqref{eq:smoothing}. Notice that this result is immediate in the case where $\para\geq2$
	since the integral is equal to $+\infty$ when $\vert n\vert\gg1$ is large enough. We therefore only consider the case
	$\para\in(1,2)$. Moreover, since the involved functions are radial, we can assume that $n\geq1$ is a positive integer and focus
	on the integral 
	\[
		e^{-2dnt}\int_0^{+\infty}e^{c\xi^{\para}}e^{-\xi^2/(2n)}\xi^{d-1}\,\mathrm d\xi.
	\]
	Let us consider the integrand
	\[
		g_n(\xi) = e^{c\xi^{\para}}e^{-\xi^2/(2n)},\quad \xi\geq0.
	\]
	Notice that this function attains its maximum on $[0,+\infty)$ at the point 
	\[
		\xi_n = (\para cn)^{\frac1{2-\para}}>0,
	\]
	and that this maximum is given by
	\[
		\Vert g_n\Vert_{\infty} 
		= g_n(\xi_n) 
		= \exp\bigg(c^{\frac2{2-\para}}\bigg(\para^{\frac{\para}{2-\para}} - \frac12\para^{\frac 2{2-\para}}\bigg)n^{\frac{\para}{2-\para}}\bigg) 
		= e^{C_{\alpha}n^{\frac{\para}{2-\para}}},
	\]
	where we set
	\[
		C_{\alpha} = c^{\frac2{2-\para}}\bigg(\para^{\frac{\para}{2-\para}} - \frac12\para^{\frac 2{2-\para}}\bigg)>0.
	\]
	Moreover, let $V_n \subset (0,+\infty)$ be a bounded neighborhood of $\xi_n$ satisfying the property
	\[
		g_n(\xi)\geq g_n(\xi_n)e^{-dnt},\quad \xi\in V_n.
	\]
	Since the function $g_n$ is non-negative, we therefore get that
	\[
		e^{-2dnt}\int_0^{+\infty}e^{c\xi^{\para}}e^{-\xi^2/(2n)}\xi^{d-1}\,\mathrm d\xi
		\geq c_n\vert V_n\vert g_n(\xi_n)e^{-3dnt} = c_n\vert V_n\vert e^{C_{\alpha}n^{\frac{\para}{2-\para}}-3dnt},
	\]
	where we set $c_n = \min_{\xi\in V_n}\xi^{d-1}>0$. Taking into account that $\xi_n$ grows only polynomially in $n$, it is not
	hard to show that the neighborhood $V_n$ can be chosen such that $\vert V_n\vert$ is uniformly bounded in $n$ and enjoys at most
	a polynomial decay with respect to $n$. In turn, the constant $c_n$ then enjoys polynomial growth.
	Moreover, since $\alpha \in (1,2)$, we have $\para/(2-\para)>1$ and this implies that
	\[
		c_n\vert V_n\vert e^{C_{\alpha}n^{\frac{\para}{2-\para}}-3dnt}\rightarrow+\infty\quad\text{as $n\rightarrow+\infty$},
	\]
	which contradicts \eqref{eq:smoothing} and, thus, proves the claim.
\end{proof}

\subsection{Case of the eigenfunctions}
Let us now focus on the (linear combinations of) eigenfunctions $\psi_{\gamma,n,m}$ of the Baouendi--Grushin operator
$\Delta_{\gamma}$, defined in \eqref{eq:eigenfunctions}. Recall that for any non-negative energy level $\lambda\geq0$, the
subspace $E_{\lambda}(-\Delta_{\gamma})$ is defined as in \eqref{eq:combilineigen}. 

We first need to derive estimates for the eigenvalue counting function for the purely discrete part of $-\Delta_\gamma$, that is, for
\[
	N(\lambda) = \#\{ (n,m) \in \ZZ^d \setminus \{0\}\times\NN \colon \lambda_{\gamma,n,m} \leq \lambda \},
\]
where $\#$ denotes the cardinality of the corresponding set.

\begin{lem}\label{lem:weyllaw} 
There exists a positive constant $c' > 0$ such that for all $\lambda\geq0$,
\[
	N(\lambda)
	\leq
	c'\lambda^{c_{\gamma}}
	\quad\text{with}\quad
	c_{\gamma}
	=
	d\biggl( \frac14+\frac1{2\gamma} \biggr)(1+\gamma)
	.
\]
\end{lem}

\begin{proof}
	Let us recall that each eigenvalue $\lambda_{\gamma,n,m}$ is given by
	$\lambda_{\gamma,n,m} = \abs{n}^{\frac{2}{1+\gamma}}\lambda_{\gamma,m}$. As a consequence, $\lambda_{\gamma,n,m} \leq \lambda$
	implies the inequalities $\lambda_{\gamma,m} \leq \lambda$ and $\abs{n} \leq (\lambda/\lambda_{\gamma,0})^{(1+\gamma)/2}$. On
	the one hand, it is classical that the cardinality of $n \in \ZZ^d\setminus\{0\}$ satisfying the latter is asymptotically
	bounded by a constant multiple of $(\lambda^{(1+\gamma)/2})^{d/2}=\lambda^{(1+\gamma)d/4}$, see, e.g.,
	\cite[Proposition~XIII.15.2]{ReedS-78}. On the other hand, the cardinality of $m \in \NN$ satisfying the former is bounded by a
	constant multiple of $\lambda^{(1+\gamma)d/(2\gamma)}$, see, e.g., \cite[Remark~5.7]{CDR} and also
	\cite[Theorem~2.3.2]{BoggiattoBR-96}. In light of $\lambda^{(1+\gamma)d/4} \lambda^{(1+\gamma)d/(2\gamma)} = \lambda^{c_\gamma}$,
	this proves the claim.
\end{proof}

We can now derive very precise quantitative smoothing estimates for (linear combinations of) eigenfunctions.

\begin{lem}\label{lem:regeigen}
	There exist some positive constants $c,c_1>0$ such that for all $(\alpha,\beta)\in\mathbb N^{2d}$, $n\in\spt\setminus\{0\}$ and
	$m\geq0$,
	\[
		\norm{ \partial^{\alpha}_x \partial^{\beta}_y \psi_{\gamma,n,m} }_{L^2(\var)}
		\leq
		c^{1+\abs{\alpha}+\abs{\beta}}\,\abs{n}^{\frac{\abs{\alpha}}{1+\gamma} + \abs{\beta}}\,(\alpha!)^{\frac{\gamma}{1+\gamma}}\,
			e^{c_1\lambda_{\gamma,m}^{\frac12 + \frac1{2\gamma}}}
		.
	\]
	Moreover, $c$ can be chosen such that for all $(\alpha,\beta) \in \NN^{2d}$, $\lambda\geq0$, and
	$f\in E_{\lambda}(-\Delta_{\gamma})$, we also have
	\[
		\norm{ \partial^{\alpha}_x \partial^{\beta}_y f}_{L^2(\var)}
		\leq
		c^{1+\abs{\alpha}+\abs{\beta}}\,\lambda^{c_{\gamma} + \frac{\abs{\alpha} + (1+\gamma)\abs{\beta}}{2}}\,
			(\alpha!)^{\frac{\gamma}{1+\gamma}}\,e^{c_1\lambda^{\frac12 + \frac1{2\gamma}}}\,\norm{f}_{L^2(\var)}
		,
	\]
	where $c_{\gamma}>0$ is the constant appearing in Lemma \ref{lem:weyllaw}.
\end{lem}

\begin{proof}
	It follows from the tensorized structure \eqref{eq:eigenfunctions} of each eigenfunction $\psi_{\gamma,n,m}$ that for all
	$(\alpha,\beta) \in \NN^{2d}$, $n\in\spt\setminus\{0\}$, and $m\geq0$, we have
	\[
		\norm{ \partial^{\alpha}_x\partial^{\beta}_y\psi_{\gamma,n,m} }_{L^2(\var)}
		=
		\norm{ \partial^{\alpha}_x(M_{\gamma,n}\phi_{\gamma,m}) }_{L^2(\RR^d)} \norm{ \partial^{\beta}_y\varphi_n }_{L^2(\TT^d)}
		.
	\]
	We have to consider the two above norms. On the one hand, it is known from the work \cite[Theorem 2.1]{Alp20b} that there exist
	positive constants $c,c_1>0$ such that for all $m\geq0$,
	\[
		\norm{ \partial^{\alpha}_x\phi_{\gamma,m} }_{L^2(\RR^d)} 
		\leq
		c^{1+\abs{\alpha}}\,(\alpha!)^{\frac{\gamma}{1+\gamma}}\,e^{c_1\lambda_{\gamma,m}^{\frac12+\frac1{2\gamma}}}
		.
	\]
	We therefore deduce from these estimates and the definition \eqref{eq:isometry} of the isometry $M_{\gamma,n}$ that for all
	$n \in \ZZ^d\setminus\{0\}$ and $m\geq0$,
	\begin{align*}
		\norm{ \partial^{\alpha}_x(M_{\gamma,n}\phi_{\gamma,m}) }_{L^2(\RR^d)}
		&=
		\abs{n}^{\frac{\abs{\alpha}}{1+\gamma}}\norm{\partial^{\alpha}_x\phi_{\gamma,m}}_{L^2(\RR^d)} \\[5pt]
		&\leq
		c^{1+\abs{\alpha}+\abs{\beta}}\abs{n}^{\frac{\abs{\alpha}}{1+\gamma}}\,(\alpha!)^{\frac{\gamma}{1+\gamma}}\,
			e^{c_1\lambda_{\gamma,m}^{\frac12 + \frac1{2\gamma}}}
		.
	\end{align*}
	On the other hand, it immediately follows from the definition \eqref{eq:fouriercoeff} of the functions $\varphi_n$ that for all
	$\beta \in \NN^d$ and all $n \in \ZZ^d\setminus\{0\}$, we have
	\[
		\norm{\partial^{\beta}_y\varphi_n}_{L^2(\TT^d)}
		=
		n^{\beta}
		\leq
		\abs{n}^{\abs{\beta}}
		.
	\]
	This ends the proof of the estimate for the eigenfunctions $\psi_{\gamma,n,m}$.

	Let us now consider $\lambda\geq0$ and $f\in E_{\lambda}(-\Delta_{\gamma})$. We expand $f$ as a linear combination of
	eigenfunctions,
	\[
		f  
		= 
		\sum_{\lambda_{\gamma,n,m}\leq\lambda}a_{n,m}\psi_{\gamma,n,m}
	\]
	with coefficients $a_{n,m}$. Using H\"older's inequality, the partial derivatives of the function $f$ can then be bounded as
	\[
		\norm{\partial^{\alpha}_x\partial^{\beta}_yf}^2_{L^2(\var)} 
		\leq
		N(\lambda)\sum_{\lambda_{\gamma,n,m}\leq\lambda}\abs{a_{n,m}}^2
			\norm{\partial^{\alpha}_x\partial^{\beta}_y\psi_{\gamma,n,m}}^2_{L^2(\var)}
		.
	\]
	Recalling that $\lambda_{\gamma,n,m} \leq \lambda$ implies the inequalities $\lambda_{\gamma,m} \leq \lambda$ and
	$\abs{n} \leq (\lambda/\lambda_{\gamma,0})^{(1+\gamma)/2}$, the claimed estimates for the function $f$ therefore easily follow
	from the ones for the single eigenfunctions $\psi_{\gamma,n,m}$ and Lemma~\ref{lem:weyllaw}.
\end{proof}

Unlike the evolution operators generated by the operator $-\Delta_{\gamma}$ and its fractional powers, the eigenfunctions
$\psi_{\gamma,n,m}$ and their linear combinations enjoy localization properties in the variable $x\in\mathbb R^d$, which are
quantified in the following lemma.

\begin{lem}\label{lem:loceigen}
	There exist positive constants $c,c_1,c_2 > 0$ such that for all $n\in\spt\setminus\{0\}$ and $m\geq0$,
	\[
		\norm{ e^{c_2\abs{n}\abs{x}^{1+\gamma}}\psi_{\gamma,n,m}}_{L^2(\var)}
		\leq
		ce^{c_1\lambda_{\gamma,m}^{\frac12 + \frac1{2\gamma}}}
		.
	\]
	Moreover, $c$ can be chosen such that for all $\lambda\geq0$ and $f\in E_{\lambda}(-\Delta_{\gamma})$, we have
	\[
		\norm{ e^{c_2\abs{x}^{1+\gamma}}f }_{L^2(\var)}
		\leq
		c\lambda^{c_{\gamma}} e^{c_1\lambda^{\frac12 + \frac1{2\gamma}}}
		,
	\]
	where $c_{\gamma}>0$ is the constant appearing in Lemma \ref{lem:weyllaw}.
\end{lem}

\begin{proof} 
	Taking into account again the tensorized structure \eqref{eq:eigenfunctions} of each eigenfunction $\psi_{\gamma,n,m}$ and the
	definition \eqref{eq:isometry} of the isometry $M_{\gamma,n}$, it follows that for all $c_2>0$, $n\in\spt\setminus\{0\}$, and
	$m\geq0$,
	\begin{align*}
		\norm{ e^{c_2\abs{n}\abs{x}^{1+\gamma}}\psi_{\gamma,n,m} }_{L^2(\var)}
		&=
		\norm{ e^{c_2\abs{n}\abs{x}^{1+\gamma}}M_{\gamma,n}\phi_{\gamma,m} }_{L^2(\RR^d)} \norm{\varphi_n }_{L^2(\TT^d)}\\
		&=
		\norm{ e^{c_2\abs{x}^{1+\gamma}}\phi_{\gamma,m} }_{L^2(\RR^d)}
		,
	\end{align*}
	where we have used that the functions $\varphi_n$ are normalized in $L^2(\TT^d)$. Moreover, it is known from the work
	\cite[Theorem 2.1]{Alp20b} that $c_2$ can be chosen such that with some positive constants $c,c_1>0$ for all $m\geq0$ we have
	\[
		\norm{ e^{c_2\abs{x}^{1+\gamma}}\phi_{\gamma,m} }_{L^2(\RR^d)}
		\leq
		ce^{c_1\lambda_{\gamma,m}^{\frac12+\frac1{2\gamma}}}
		.
	\]
	This proves the claim for the single eigenfunctions $\psi_{\gamma,n,m}$. Taking into account the bound
	$e^{c_2\abs{x}^{1+\gamma}} \leq e^{c_2\abs{n}\abs{x}^{1+\gamma}}$ for all $n \in \spt\setminus\{0\}$, the case of the
	functions in $E_{\lambda}(-\Delta_{\gamma})$ is finally treated analogously as in the proof of Lemma~\ref{lem:regeigen}.
\end{proof}

\subsection{Boundary case}\label{subsec:bd}
To end this section, let us consider the Baouendi--Grushin operator 
\[
	\Delta^{bd}_{\gamma} = \Delta_x + \vert x\vert^{2\gamma}\Delta_y,\quad (x,y)\in\mathbb R^d\times(0,2\pi)^d,
\]
where $\Delta_y$ denotes in this case the Laplacian on the hypercube $(0,2\pi)^d$ with Dirichlet boundary conditions. In the
following, we explain how the results of this section fit for this operator. 

First of all, recall that the spectrum of the operator $-\Delta_y$ is given by 
\[
	\sigma(-\Delta_y)
	=
	\biggl\{ \frac{\abs{n}^2}{4} \colon n\in(\NN^*)^d \biggr\}
	.
\]
Moreover, for each $n\in(\NN^*)^d$, let $\varphi_n\in L^2((0,2\pi)^d)$ be given by
\[
	\varphi_n(y)
	=
	\pi^{-\frac d2}\prod_{j=1}^d\sin\biggl(\frac{n_jy_j}{2}\biggr)
	,\quad
	y\in(0,2\pi)^d
	.
\]
Clearly, each such function $\varphi_n$ is a normalized eigenfunction of the operator $-\Delta_y$ associated with the eigenvalue
$\abs{n}^2/4$, and the family $(\varphi_n)_n$ forms a Hilbert basis of the space $L^2((0,2\pi)^d)$. Correspondingly, the operator
$\Delta^{bd}_{\gamma}$ is transformed as
\[
	\Delta_x + \abs{x}^{2\gamma}\Delta_y
	\,\rightsquigarrow\,
	\Delta_x - \frac{\abs{n}^2}{4}\abs{x}^{2\gamma}
	,
\]
and respective variants of the diagonalization formula \eqref{eq:diag}, the eigenvalues $\lambda_{\gamma,n,m}$ of
$\Delta^{bd}_{\gamma}$, and of the associated eigenfunctions $\psi_{\gamma,n,m}$ in \eqref{eq:eigenfunctions} hold.
Since also
\[
	\forall n \in (\NN^*)^d,\ \alpha \in \NN^d
	,\quad
	\norm{ \partial^{\alpha}_y\varphi_n }_{L^2((0,2\pi)^d)}
	=
	\frac{n^{\alpha}}{2^{\abs{\alpha}}}
	\leq
	\biggl( \frac{\abs{n}}{2} \biggr)^{\abs{\alpha}}
	,
\]
corresponding variants of Lemmas~\ref{lem:regy}, \ref{lem:regxbg} and~\ref{lem:optiregx} for the evolution operators generated by
$\Delta^{bd}_{\gamma}$ and of Lemmas~\ref{lem:regeigen} and~\ref{lem:loceigen} for its eigenfunctions can be proved in exactly
the same fashion.

\section{Unique continuation estimates for the Baouendi--Grushin equations}\label{sec:mainProofs}
In this section, we prove unique continuation estimates for the semigroups generated by (fractional) Baouendi--Grushin operators
by combining the abstract framework from Section~\ref{sec:uniqueCont} with the smoothing and localizing properties established in
Section~\ref{sec:Grushin}. As a byproduct, we also obtain spectral inequalities for the Baouendi--Grushin operators from thick
sensor sets, and prove two results regarding their optimality. Finally, we derive spectral inequalities for the (linear
combinations of) eigenfunctions of the Baouendi--Grushin operators.

\subsection{Unique continuation from thick sets}
In this first subsection, we prove Theorem~\ref{thm:ucthickbg} and Corollary~\ref{cor:speinthickbg}.

\begin{proof}[Proof of Theorem~\ref{thm:ucthickbg} and Corollary~\ref{cor:speinthickbg}]
	It follows from Corollary~\ref{cor:smoothingBG} that the function $f = \euler^{-(t_x+t_y)(-\Delta_\gamma)^{(1+\gamma)/2}}g$
	for $t_x \in (0,t_0)$, $t_y > 0$ satisfies the smoothing estimate \eqref{eq:genSmoothing} with
	$D = c\norm{g}_{L^2(\var)}/(t_x)^{d/\gamma}$, $c_x = c/(t_x)^{\gamma/(\gamma+1)}$, $c_y = c/t_y$, and $\mu = 1$. In light of the
	hypothesis that $\omega$ is $(\theta,L_x,L_y)$-thick in $\var$, we are therefore in the setting of Example~\ref{ex:ucThick}.
	Working with $\eps t_x^{2d/\gamma} / c^2$ instead of $\eps$ there, we conclude that there exists a constant $K = K(d,\gamma) > 0$
	such that for all $t_x \in (0,t_0)$, $t_y > 0$, $\eps > 0$, and $g\in L^2(\var)$, we have
	\[
		\norm{ \euler^{-(t_x+t_y)(-\Delta_{\gamma})^{\frac{1+\gamma}2}}g }_{L^2(\var)}^2
		\leq
		\biggl( \frac{K}{\theta} \biggr)^{KC_{\eps,\gamma,t_x,t_y}}
			\norm{ \euler^{-(t_x+t_y)(-\Delta_{\gamma})^{\frac{1+\gamma}2}}g }_{L^2(\omega)}^2 + \eps\norm{g}_{L^2(\var)}^2
	\]
	with
	\[
		C_{\eps,\gamma,t_x,t_y}
		=
		(1-\log\eps-\log t_x)\exp\biggl( \frac{KL_x}{(t_x)^{\frac{\gamma}{\gamma+1}}} + \frac{KL_y}{t_y} \biggr)
		,
	\]
	upon a suitable adaptation of the constant $K$, especially in order to absorb the constant $c$. The claim of
	Theorem~\ref{thm:ucthickbg} with $t_0/2$ instead of $t_0$ then follows with the particular choice $t_x = t_y = t \in (0,t_0/2)$.

	In order to derive Corollary~\ref{cor:speinthickbg}, we take specifically
	$g = \euler^{(t_x+t_y)(-\Delta_\gamma)^{\frac{1+\gamma}{2}}}f \in L^2(\var)$ with
	$f \in \cE_\lambda(-\Delta_\gamma) = \cE_{\lambda^{\frac{1+\gamma}{2}}}((-\Delta_\gamma)^{\frac{1+\gamma}{2}})$ and any choices
	of $t_x \in (0,t_0)$ and $t_y > 0$. We then have by functional calculus that
	\[
		\norm{g}_{L^2(\var)}^2
		\leq
		\euler^{2(t_x+t_y)\lambda^{\frac{1+\gamma}{2}}} \norm{f}_{L^2(\var)}^2
		.
	\]
	Choosing $\eps = \frac{1}{2}\euler^{-2(t_x+t_y)\lambda^\frac{1+\gamma}{2}}$, the claim of Corollary~\ref{cor:speinthickbg}
	follows from the above after reordering the terms.

	The stronger statements in the case $\gamma = 1$ are an easy modification of the above since then Corollary~\ref{cor:smoothingBG}
	guarantees that $f = \euler^{(t_x+t_y)\Delta_\gamma}g$ satisfies \eqref{eq:genSmoothing} for all $t_x,t_y > 0$ and with instead
	$D = c\norm{g}_{L^2(\var)}$. As a consequence, the term $\log t_x$ is removed from $C_{\eps,1,t_x,t_y}$. In the context of
	Corollary~\ref{cor:speinthickbg}, we then finally choose $t_x = L_x^2$ and $t_y = L_y$. This completes the proof.
\end{proof}

\subsection{Optimality}
In this subsection, we tackle the proof of the optimality results stated in Section \ref{subsubsection:thckness}, namely,
Propositions~\ref{prop:geometry} and~\ref{prop:power}. First of all, we check that the thickness is a necessary geometric
condition in order to get spectral inequalities for the Baouendi--Grushin operators.

\begin{proof}[Proof of Proposition \ref{prop:geometry}] 
	Consider a measurable subset $\omega\subset\mathbb R^d\times\mathbb T^d$ with positive measure and a non-negative energy level
	$\lambda\geq0$. Suppose that for some constant $c_{\lambda,\omega} > 0$ we have the spectral inequality
	\begin{equation}\label{eq:hypoinspe}
		\forall f\in\mathcal E_{\lambda}(-\Delta_{\gamma})
		,\quad
		\Vert f\Vert_{L^2(\mathbb R^d\times\mathbb T^d)}\le c_{\lambda,\omega}\Vert f\Vert_{L^2(\omega)}.
	\end{equation}
	We aim at proving that the set $\omega$ is thick in the sense of Definition \ref{def:thickStrip}. Since $0$ is an eigenvalue of
	the Laplace operator $\Delta_y$ on the $d$-dimensional torus $\mathbb T^d$, we have the inclusion 
	\begin{equation}\label{eq:inclusionspecsub}
		\mathcal E_{\lambda}(-\Delta_x)\subset\mathcal E_{\lambda}(-\Delta_{\gamma}),
	\end{equation}
	where $\mathcal E_{\lambda}(-\Delta_x)$ denotes the spectral subspace of the operator $-\Delta_x$ on the whole Euclidean space
	$\mathbb R^d$ associated with the energy level $\lambda$, that is,
	\[
		\mathcal E_{\lambda}(-\Delta_x) 
		= \mathbbm1_{[0,\lambda]}(-\Delta_x) 
		= \Bigl\{f\in L^2(\mathbb R^d) : \Supp(\mathscr Ff)\subset \overline{B(0,\sqrt{\lambda})}\Bigr\},
	\]
	where $\mathscr F$ denotes the Fourier transform. The following is inspired by \cite[p.113]{HavinJ-94}. Let us consider the
	function $f\in L^2(\mathbb R^d)$ defined by $f=\mathscr F^{-1}(\mathbbm1_{[-c\sqrt\lambda,c\sqrt\lambda]^d})$, where $c>0$ is
	chosen such that $[-c\sqrt\lambda,c\sqrt\lambda]^d\subset B(0,\sqrt{\lambda})$ and where $\mathscr F^{-1}$ denotes the inverse
	Fourier transform. Fixing $x_0\in\mathbb R^d$, we also consider the translate $f_{x_0}\in L^2(\mathbb R^d)$ of $f$ defined by
	\[
		f_{x_0}(x) = f(x-x_0),\quad x\in\mathbb R^d.
	\]
	Since the Fourier transform of the function $f_{x_0}$ coincides with the one of the function $f$, up to a complex factor of
	modulus $1$, we also have $f_{x_0}\in\mathcal E_{\lambda}(-\Delta_x)$. Notice that every function $g \in L^2(\RR^d)$ can be
	treated as a function in $L^2(\mathbb R^d\times\mathbb T^d)$ that is constant with respect to the $\mathbb T^d$-variable.
	In this sense, recalling the inclusion \eqref{eq:inclusionspecsub} of the spectral subspaces, we deduce that
	$f_{x_0}\in\mathcal E_{\lambda}(-\Delta_{\gamma})$. As a consequence of the spectral inequality \eqref{eq:hypoinspe}, we then
	obtain
	\[
		\Vert f\Vert^2_{L^2(\mathbb R^d)} = \Vert f_{x_0}\Vert^2_{L^2(\mathbb R^d\times\mathbb T^d)}
		\le c^2_{\lambda,\omega}\Vert f_{x_0}\Vert^2_{L^2(\omega)} 
		= c^2_{\lambda,\omega}\int_{\mathbb T^d}\Vert f_{x_0}\Vert^2_{L^2(\omega_y)}\,\mathrm dy,
	\]
	where we set
	\[
		\omega_y = \big\{ x \in \RR^d \colon (x,y) \in \omega\big\},\quad y\in\mathbb T^d.
	\]
	Given some $y\in\mathbb T^d$, we now have to control the norm $\norm{f_{x_0}}_{L^2(\omega_y)}$. To this end, let
	$L= L_{\lambda,\omega}>0$ be a radius whose value will be adjusted later. We split the norm into two parts,
	\begin{align*}
		\Vert f_{x_0}\Vert^2_{L^2(\omega_y)}
		& = \int_{(\omega_y-x_0)\cap[-L,L]^d}\vert f(x)\vert^2\,\mathrm dx + \int_{(\omega_y-x_0)\cap([-L,L]^d)^c}\vert f(x)\vert^2\,\mathrm dx \\[5pt]
		& \le \int_{(\omega_y-x_0)\cap[-L,L]^d}\vert f(x)\vert^2\,\mathrm dx + \int_{\vert x\vert>L}\vert f(x)\vert^2\,\mathrm dx.
	\end{align*}
	On the one hand, since $f\in L^2(\mathbb R^d)$, the dominated convergence theorem implies that the radius $L\gg1$ can be chosen
	large enough such that
	\[
		c^2_{\lambda,\omega}\int_{\vert x\vert>L}\vert f(x)\vert^2\,\mathrm dx\le\frac12\Vert f\Vert^2_{L^2(\mathbb R^d)}.
	\]
	On the other hand, the Fourier inversion formula, the Cauchy-Schwarz inequality, and Plancherel's theorem imply that
	\[
		\Vert f\Vert_{L^{\infty}(\mathbb R^d)}
		\le\frac1{(2\pi)^d}\Vert\mathscr Ff\Vert_{L^1(\mathbb R^d)}
		\le\frac{(2c\sqrt{\lambda})^{\frac d2}}{(2\pi)^d}\Vert\mathscr Ff\Vert_{L^2(\mathbb R^d)}
		= \bigg(\frac{c\sqrt{\lambda}}{\pi}\bigg)^{\frac d2}\Vert f\Vert_{L^2(\mathbb R^d)}.
	\]
	We therefore deduce that
	\begin{align*}
		\int_{(\omega_y-x_0)\cap[-L,L]^d}\vert f(x)\vert^2\,\mathrm dx
		& \le\vert(\omega_y-x_0)\cap[-L,L]^d\vert\Vert f\Vert^2_{L^{\infty}(\mathbb R^d)} \\[5pt]
		& \le\vert\omega_y\cap(x_0+[-L,L]^d)\vert\bigg(\frac{c\sqrt{\lambda}}{\pi}\bigg)^d\Vert f\Vert^2_{L^2(\mathbb R^d)}
	\end{align*}
	since the Lebesgue measure is invariant by translation. Gathering the above estimates, we conclude that there exists a ratio
	$\theta = \theta_{\lambda,\omega}\in(0,1]$ such that for all $x_0\in\mathbb R^d$,
	\[
		\vert\omega\cap((x_0+[-L,L]^d)\times\mathbb T^d)\vert 
		= \int_{\mathbb T^d}\vert\omega_y\cap(x_0+[-L,L]^d)\vert\,\mathrm dy
		\geq\theta(2L)^d,
	\]
	that is, the set $\omega$ is thick in the sense of Definition~\ref{def:thickStrip}.
\end{proof}

We now check that the power $\lambda^{(1+\gamma)/2}$ appearing in the spectral inequalities stated in
Corollary~\ref{cor:speinthickbg} is optimal. The following proof follows the same strategy as the one of \cite[Theorem~2.17\,$(i)$
and Theorem~2.19]{AlphonseS}.

\begin{proof}[Proof of Proposition \ref{prop:power}]
	Let $\omega\subset\mathbb R^d\times\mathbb T^d$ be a measurable set with positive measure satisfying the geometric condition
	$\overline\omega\cap\{x=0\} = \emptyset$. Recall that we aim at constructing a sequence $((\lambda_n,\psi_n))_n$ of eigenpairs
	of the operator $-\Delta_{\gamma}$, with $\lambda_n\rightarrow+\infty$, such that
	\begin{equation}\label{eq:eigenLowerBound}
		\Vert\psi_n\Vert_{L^2(\mathbb R^d\times\mathbb T^d)}\geq c_0e^{c_1\lambda_n^{\frac{1+\gamma}2}}\Vert\psi_n\Vert_{L^2(\omega)}.
	\end{equation}
	To this end, let $\lambda_{\gamma}>0$ be the smallest eigenvalue of the anharmonic oscillator $H_{\gamma}$, and let
	$\phi_{\gamma}\in L^2(\mathbb R^d)$ be an associated normalized eigenfunction. For each $n \in \ZZ^d \setminus \{0\}$, let us
	also consider the function $\psi_{\gamma,n,0} \in L^2(\mathbb R^d\times\mathbb T^d)$ defined as in \eqref{eq:eigenfunctions} by
	\[
		\psi_{\gamma,n,0}(x,y) = e^{in\cdot y} (M_{\gamma,n} \phi_\gamma)(x),\quad(x,y) \in\mathbb R^d\times\mathbb T^d,
	\]
	where $M_{\gamma,n}$ denotes the unitary transform on $L^2(\mathbb R^d)$ given by \eqref{eq:isometry}. Recall from
	Section~\ref{subsubsec:eigenfunctions} that the function $\psi_{\gamma,n,0}$ is a normalized eigenfunction of the operator
	$-\Delta_{\gamma}$, associated with the eigenvalue $\lambda_n := \vert n\vert^{2/(1+\gamma)}\lambda_{\gamma,0}$. Moreover, it is
	clear that
	\begin{equation}\label{eq:goodpower}
		\Vert\psi_{\gamma,n,0}\Vert_{L^2(\omega)} = \Vert\phi_\gamma\Vert_{L^2(\omega_n)}
		\quad\text{where}\quad
		\omega_n = \big\{(\vert n\vert^{1/(1+\gamma)}x , y) : (x,y) \in \omega\big\},
	\end{equation}
	and where $\phi_\gamma$ is interpreted as a function in $L^2(\mathbb R^d\times\mathbb T^d)$ that is constant with respect to the
	$y$-variable. Recall from \cite[Theorem 3.3]{BS} that the eigenfunction $\phi_{\gamma}$ satisfies the decay property
	\[
		\big\Vert e^{\varepsilon\vert x\vert^{1+\gamma}}\phi_\gamma\big\Vert_{L^2(\mathbb R^d)}\leq c_{\varepsilon,\gamma},
	\]
	where $\varepsilon>0$ and $c_{\varepsilon,\gamma}>0$ are positive constants. Thus, setting $L := \dist(0,\omega)>0$, we deduce
	that for all $n \in \ZZ^d \setminus \{0\}$,
	\[
		\Vert\phi_\gamma\Vert_{L^2(\omega_n)} 
		= \big\Vert e^{-\varepsilon\vert x\vert^{1+\gamma}}e^{\varepsilon\vert x\vert^{1+\gamma}} \phi_\gamma\big\Vert_{L^2(\omega_n)}
		\leq c_{\varepsilon,\gamma}e^{-\varepsilon\vert n\vert L^{1+\gamma}}.
	\]
	Taking into account that $\psi_n = \psi_{\gamma,n,0}\in L^2(\mathbb R^d\times\mathbb T^d)$ is normalized, combining this
	estimate with \eqref{eq:goodpower} and writing
	\[
		\vert n\vert L^{1+\gamma}
		= \bigg(\frac L{\sqrt{\lambda_{\gamma,0}}}\bigg)^{1+\gamma}(\vert n\vert^{\frac2{1+\gamma}}\lambda_{\gamma,0})^{\frac{1+\gamma}2} 
		= \bigg(\frac L{\sqrt{\lambda_{\gamma,0}}}\bigg)^{1+\gamma}\lambda_n^{\frac{1+\gamma}2},
	\]
	we deduce that \eqref{eq:eigenLowerBound} holds, which proves the claim.
\end{proof}

\subsection{Eigenfunctions}\label{subsec:eigenfunctionsproofs}
In this subsection, we give the proofs of Theorems~\ref{thm:ucsingleigenfunction} and~\ref{thm:ucgeneraleigenfunctions} regarding
(linear combinations of) eigenfunctions of the operator $\Delta_{\gamma}$.

\begin{proof}[Proof of Theorem \ref{thm:ucsingleigenfunction}]
Take $R_\omega > 0$ such that $\theta_\omega := \vert\omega \cap (-R_\omega,R_\omega)^d \times \TT^d\vert > 0$, and fix
$n \in \ZZ^d \setminus \{0\}$ and $m \geq 0$. It follows from Lemmas~\ref{lem:regeigen} and~\ref{lem:loceigen} that the
eigenfunction $\psi_{\gamma,n,m}$ satisfies
\begin{equation}\label{eq:smoothingAnisotropic}
	\forall \alpha,\beta \in \NN^d
	,\quad
	\norm{ \partial_x^\alpha \partial_y^\beta \psi_{\gamma,n,m} }_{L^2(\var)}
	\leq
	D c_x^{\abs{\alpha}} c_y^{\abs{\beta}} (\alpha!)^{\mu_x} (\beta!)^{\mu_y}
\end{equation}
where
\[
	D
	=
	c\euler^{c_1\lambda_{\gamma,m}^{\frac{1+\gamma}{2\gamma}}}
	,\quad
	c_x
	=
	c\abs{n}^{\frac{1}{1+\gamma}}
	,\quad
	\mu_x
	=
	\frac{\gamma}{1+\gamma}
	,\quad
	c_y
	=
	c\abs{n}
	,\quad
	\mu_y
	=
	0
	,
\]
and also \eqref{eq:genDecay} with
\[
	c_0
	=
	c_2\abs{n}
	,\quad
	\nu
	=
	\frac{1}{1+\gamma}
	.
\]
Note that \eqref{eq:smoothingAnisotropic} is a variant of \eqref{eq:genSmoothing} that takes into account the different smoothing
properties with respect to the $x$ and $y$ coordinates.
In turn, upon replacing Corollary~\ref{cor:localEstimatewChangofVariables} in the general framework by
Proposition~\ref{prop:localEstimatewChangofVariablesAnisotropic} and adapting Lemma~\ref{lem:goodAndBad} and
Corollary~\ref{cor:goodAndBad} accordingly, we conclude as in Example~\ref{ex:ucDecay} that for every $R \geq R_\omega$ we have
\begin{equation}\label{eq:ucEigen}
	\norm{\psi_{\gamma,n,m}}_{L^2(\RR^d \times \TT^d)}^2
	\leq
	\biggl( \frac K{\theta_R} \biggr)^{KC_R} \norm{\psi_{\gamma,n,m}}_{L^2(\omega)}^2 +
		2c^2\euler^{2c_1\lambda_{\gamma,m}^{\frac{1+\gamma}{2\gamma}}}\euler^{-2c_2\abs{n}R^{1+\gamma}}
\end{equation}
with
\begin{equation}\label{eq:ratioproof}
	\theta_R
	=
	\frac{\vert\omega \cap (-R,R)^d \times \TT^d\vert}{\vert(-R,R)^d \times \TT^d\vert}
	\geq
	\frac{\theta_\omega}{(2R)^d}
	>
	0
	,
\end{equation}
and
\[
	C_R
	=
	1 + \abs{n}R^{1+\gamma}  + \abs{n}R^{1+\gamma} + \abs{n}
	,
\]
and where $K > 0$ depends only on $\gamma$ and the dimension $d$.

Let us now consider $p > 1$ satisfying
\[
	c^2\euler^{2c_1(1-p)\lambda_{\gamma,0}^{\frac{1+\gamma}{2\gamma}}}
	\leq
	\frac{1}{2}
	.
\]
Note that $p$ depends only on $\gamma$, $c$, and $c_1$. We distinguish two cases and write $a\lesssim b$ for $a,b > 0$ if
the quotient $a/b$ is bounded by a constant depending at most on $\omega$, $\gamma$, and the dimension $d$.

\noindent\textit{$\triangleright$ Case 1: $c_2\abs{n}R_\omega^{1+\gamma} > pc_1\lambda_{\gamma,m}^{\frac{1+\gamma}{2\gamma}}$.} 
In this situation, we take $R = R_\omega$ in \eqref{eq:ucEigen}. We
then have
\begin{equation}\label{eq:boundUCErr}
	c^2\euler^{2c_1\lambda_{\gamma,m}^{\frac{1+\gamma}{2\gamma}}}\euler^{-2c_2\abs{n}R^{1+\gamma}}
	\leq
	c^2\euler^{2c_1(1-p)\lambda_{\gamma,m}^{\frac{1+\gamma}{2\gamma}}}
	\leq
	c^2\euler^{2c_1(1-p)\lambda_{\gamma,0}^{\frac{1+\gamma}{2\gamma}}}
	\leq
	\frac12
	,
\end{equation}
by the choice of $p$. Moreoever, we have
$\abs{n} \lesssim \abs{n}R_{\omega}^{1+\gamma} \lesssim \lambda_{\gamma,n,m}^{\frac{1+\gamma}{2}}$ and, thus,
\[
	C_{R_{\omega}}
	\lesssim
	\lambda_{\gamma,n,m}^{\frac{1+\gamma}{2}}
	.
\]
Since also
\[
	\log\biggl( \frac{K}{\theta_{R_{\omega}}} \biggr)
	\lesssim
	\log(1+R_{\omega})
	\lesssim
	1
	\lesssim
	\log(1+\lambda_{\gamma,0,0})
	\leq
	\log(1+\lambda_{\gamma,n,m})
	,
\]
inequality \eqref{eq:ucEigen} is in this case consistent with the claimed inequality for $\psi_{\gamma,n,m}$, upon a suitable
adaptation of the constant $K$.

\noindent\textit{$\triangleright$ Case 2:  $c_2\abs{n}R_\omega^{1+\gamma} \leq pc_1\lambda_{\gamma,m}^{\frac{1+\gamma}{2\gamma}}$.} 
In this case, there exists some $R_0 \geq R_\omega$ with
\[
	c_2\abs{n}R_0^{1+\gamma}
	=
	pc_1\lambda_{\gamma,m}^{\frac{1+\gamma}{2\gamma}}
	.
\]
Choosing $R = R_0$ in \eqref{eq:ucEigen}, the inequality \eqref{eq:boundUCErr} is still valid, and we have
\[
	\abs{n}
	\lesssim
	\abs{n}R_\omega^{1+\gamma}
	\leq
	\abs{n}R_0^{1+\gamma}
	=
	\frac{pc_1}{c_2}\,\lambda_{\gamma,m}^{\frac{1+\gamma}{2\gamma}}
	\leq
	\frac{pc_1}{c_2}\,\lambda_{\gamma,0}^{\frac{1+\gamma}{2\gamma}}\ \biggl( \frac{\lambda_{\gamma,n,m}}{\lambda_{\gamma,0}}
		\biggr)^{\frac{1+\gamma}{2\gamma}}
	\lesssim
	\lambda_{\gamma,n,m}^{\frac{1+\gamma}2}
	,
\]
since $\gamma\geq1$, so that again
\[
	C_{R_0}
	\lesssim
	\lambda_{\gamma,n,m}^{\frac{1+\gamma}{2}}
	.
\]
At the same time, $R_0$ satisfies
\[
	R_0^{1+\gamma}
	\leq
	\abs{n}R_0^{1+\gamma}
	\lesssim
	\lambda_{\gamma,n,m}^{\frac{1+\gamma}2}
	,
\]
so that
\[
	\log\biggl( \frac{K}{\theta_{R_0}} \biggr)
	\lesssim
	\log(1+\lambda_{\gamma,n,m})
	.
\]
We therefore conclude that also in this case \eqref{eq:ucEigen} is consistent with the claim for $\psi_{\gamma,n,m}$, again
with a suitable adaptation of the constant $K$. This ends the proof of Theorem~\ref{thm:ucsingleigenfunction}.
\end{proof}

Let us finally turn to linear combinations of eigenfunctions of the operator $\Delta_{\gamma}$.

\begin{proof}[Proof of Theorem~\ref{thm:ucgeneraleigenfunctions}]
	Let $R_{\omega}$ be as in the proof of Theorem~\ref{thm:ucsingleigenfunction} above, and fix some
	$f\in E_{\lambda}(-\Delta_{\gamma})$. It then follows from Lemmas~\ref{lem:regeigen} and~\ref{lem:loceigen} that $f$ satisfies
	\eqref{eq:smoothingAnisotropic} and \eqref{eq:genDecay} with instead
	\[
		D
		=
		c\lambda^{c_\gamma}\euler^{c_1\lambda^{\frac{1+\gamma}{2\gamma}}}\norm{f}_{L^2(\var)}
		,\quad
		c_x
		=
		c\sqrt{\lambda}
		,\quad
		\mu_x
		=
		\frac{\gamma}{1+\gamma}
		,\quad
		c_y
		=
		c\lambda^{\frac{1+\gamma}2}
		,\quad
		\mu_y
		=
		0
	\]
	and
	\[
		c_0
		=
		c_2
		,\quad
		\nu
		=
		\frac{1}{1+\gamma}
		.
	\]
	We then conclude in the same way as in the proof of Theorem~\ref{thm:ucsingleigenfunction} that for every $R \geq R_\omega$ we
	have
	\[
		\norm{f}_{L^2(\RR^d \times \TT^d)}^2
		\leq
		\biggl( \frac K{\theta_R} \biggr)^{KC_R} \norm{f}_{L^2(\omega)}^2 +
			c^2\lambda^{2c_\gamma}\euler^{2c_1\lambda^{\frac{1+\gamma}{2\gamma}}}\euler^{-2c_2R^{1+\gamma}}
			\norm{f}_{L^2(\RR^d \times \TT^d)}^2
		,
	\]
	with $\theta_R$ given by \eqref{eq:ratioproof} and
	\[
		C_R
		=
		1 + R^{1+\gamma}  + \lambda^{\frac{1+\gamma}2}R^{1+\gamma} + \lambda^{\frac{1+\gamma}2}
		,
	\]
	and where $K > 0$ depends only on $\gamma$ and the dimension $d$. Now, we choose $R$ such that
	\[
		c^2\lambda^{2c_\gamma}\euler^{2c_1\lambda^{\frac{1+\gamma}{2\gamma}}}\euler^{-2c_2R^{1+\gamma}}
		\le
		\frac12
		.
	\]
	We obviously have the bound
	\[
		R^{1+\gamma}
		\lesssim
		\lambda^{\frac{1+\gamma}{2\gamma}}
		,
	\]
	so that
	\[
		C_R
		\lesssim
		\lambda^{(\frac12+\frac1{2\gamma})(1+\gamma)}
	\]
	and
	\[
		\log(K/\theta_R)
		\lesssim
		\log(1+\lambda)
		.
	\]
	This proves	the claim.
\end{proof}

\section{The anisotropic Shubin operators}\label{sec:Shubin}

In this last section, we apply the general framework from Section~\ref{sec:uniqueCont} to the evolution equations associated to
the anisotropic Shubin operators, given by
\[
	H_{k,m} = (-\Delta)^m + \vert x\vert^{2k},\quad x\in\mathbb R^d,
\]
where $k,m\geq1$ are positive integers. Recall (e.g. from \cite[Section 2]{Alp20b}) that the operators $H_{k,m}$ equipped with the
domains
\[
	D(H_{k,m}) = \big\{g\in L^2(\mathbb R^d) : H_{k,m}g\in L^2(\mathbb R^d)\big\},
\]
generate strongly continuous semigroups $(e^{-tH_{k,m}})_{t\geq0}$ on $L^2(\mathbb R^n)$. It is also known from the work
\cite{Alp20b} that the associated evolution operators $e^{-tH_{k,m}}$ enjoy very strong smoothing and localizing properties in
Gelfand--Shilov spaces, for precise estimates see \eqref{eq:smoothshubin} and \eqref{eq:locshubin} below, which allow us to derive
quantitative unique continuation estimates for them. As a byproduct, we obtain positive cost-uniform approximate
null-controllability results for the evolution equations associated to the operators $H_{k,m}$, which are qualitatively new in the
particular case of $k = m = 1$.

\subsection{Unique continuation estimates}

The aforementioned smoothing and localizing result for the evolution operators generated by $H_{k,m}$ is precisely
\cite[Theorem 2.3]{Alp20b}. It states that there exist positive constants $c,c_1>0$ and $t_0\in(0,1)$ such that for all
$\alpha\in\mathbb N^d$, $t\in(0,t_0)$, and $g\in L^2(\mathbb R^d)$,
\begin{equation}\label{eq:smoothshubin}
	\big\Vert\partial^{\alpha}_x(e^{-tH_{k,m}}g)\big\Vert_{L^2(\mathbb R^d)}
	\le
	\frac{c^{1+\vert\alpha\vert}}{t^{\frac{k\vert\alpha\vert}{k+m}}}\,(\alpha!)^{\frac k{k+m}}\,\Vert g\Vert_{L^2(\mathbb R^d)},
\end{equation}
and
\begin{equation}\label{eq:locshubin}
	\big\Vert e^{c_1t\vert x\vert^{1+\frac km}}(e^{-tH_{k,m}}g)\big\Vert_{L^2(\mathbb R^d)}
	\le
	c\Vert g\Vert_{L^2(\mathbb R^d)}.
\end{equation}
These enable us to obtain the following quantitative unique continuation estimates.

\begin{thm}\label{thm:ucShubinDecay}
	Let $\omega\subset\mathbb R^d$ be measurable with positive measure, and let $\scrit_{\omega}\in[0,+\infty)$ be the radius
	defined by
	\begin{equation}\label{eq:distancezero}
		\scrit_{\omega} = \inf\bigl\{\scrit>0 : \vert\omega\cap(-R,R)^d\vert>0\bigr\}.
	\end{equation}
	There exist some positive constants $c,c_1,K>0$ and $t_0\in(0,1)$ such that for all $t\in(0,t_0)$, $\scrit>\scrit_{\omega}$, and
	$g\in L^2(\mathbb R^d)$,
	\[
		\big\Vert e^{-tH_{k,m}}g\big\Vert^2_{L^2(\mathbb R^d)}
		\leq
		\biggl( \frac K{\theta_{\scrit}} \biggr)^{K(1+t\scrit^{1+\frac km}+\scrit^{1+\frac km}/t^{\frac km})}
			\big\Vert e^{-tH_{k,m}}g\big\Vert^2_{L^2(\omega)} + ce^{-c_1t\scrit^{1+\frac km}}\Vert g\Vert^2_{L^2(\mathbb R^d)},
	\]
	where the ratio $\theta_{\scrit}\in(0,1]$ is given by
	\begin{equation}\label{eq:ratio}
		\theta_{\scrit} = \frac{\vert\omega\cap(-R,R)^d\vert}{(2R)^d}.
	\end{equation}
\end{thm}

\begin{proof}
	In line of Remark~\ref{rk:ucDecay}\,(2), we infer from \eqref{eq:smoothshubin} and \eqref{eq:locshubin} that each
	$f = e^{-tH_{k,m}}g$ satisfies the $\RR^d$ variants of \eqref{eq:genSmoothing} and \eqref{eq:genDecay} with
	\[
		D = c\norm{g}_{L^2(\RR^d)}
		,\quad
		c_x = \frac{c}{t^{\frac{k}{k+m}}}
		,\quad
		\mu = \frac{k}{k+m}
	\]
	and
	\[
		c_0 = c_1t
		,\quad
		\nu = \frac{m}{k+m}
		.
	\]
	The claim with $c$ and $c_1$ replaced by $2c^2$ and $2c_1$, respectively, therefore follows just as in Example~\ref{ex:ucDecay},
	and it only remains to relabel the constants accordingly.
\end{proof}

\subsection{Approximate null-controllability}

The quantitative estimates stated in Theorem~\ref{thm:ucShubinDecay} have consequences on the cost-uniform approximate
null-controllability properties of the evolution equations associated with the operators $H_{k,m}$, given by
\begin{equation}\label{eq:Shubin}\tag{$E_{k,m}$}
\begin{cases}
	\partial_tf(t,x) + H_{k,m}f(t,x) = h(t,x)\bmone_{\omega}(x), & t>0,\ x\in\RR^d,\\
	f(0,\cdot) = f_0\in L^2(\RR^d),
\end{cases}
\end{equation}
where $\omega\subset\mathbb R^d$ is a measurable set with positive measure and $h\in L^2((0,T)\times\omega)$ is a control. The
notions of exact and cost-uniform approximate null-controllability for the equation \eqref{eq:contbg} from
Definitions~\ref{dfn:exactcontrol} and~\ref{dfn:approxunifcost} and their interpretation in terms of (weak) observability
estimates in \eqref{eq:exactobs} and Proposition~\ref{prop:dualityapproxcontuc}, respectively, carry over to the equation
\eqref{eq:Shubin} verbatim. With this in mind, a straightforward consequence of Theorem~\ref{thm:ucShubinDecay} is the following
null-controllability result.

\begin{cor}\label{cor:exactcontshubin}
	The evolution equation \eqref{eq:Shubin} is cost-uniformly approximately null-controllable from every measurable set
	$\omega\subset\mathbb R^d$ with positive measure and in every positive time $T>0$.
\end{cor}

\begin{rk}
	The exact null-controllability properties of the equation \eqref{eq:Shubin} (and its fractional counterparts) have been widely
	studied, see, e.g., \cite{BeauchardJPS-21,BeauchardPS-18,DickeSV-23,MartinPS-22,M} for the specific case $k=m=1$, and
	\cite{AlphonseS,DickeS-22,DickeSV, M2,Martin-22} where the general case $k,m\geq1$ is considered. In this regard,
	Corollary~\ref{cor:exactcontshubin} does not give a qualitatively new result when $(k,m) \neq (1,1)$. In fact, we know from
	\cite[Theorem 2.5]{M2} that when $(k,m)\ne(1,1)$, then the evolution equation \eqref{eq:Shubin} is even exactly
	null-controllable in every positive time $T>0$ and from every measurable control support $\omega\subset\mathbb R^d$ with
	positive measure. In the case $k=m=1$, however, the exact null-controllability properties of the harmonic heat equation
	$($\hyperref[eq:Shubin]{$E_{1,1}$}$)$ are not yet fully understood. On the one hand, it is known from \cite[Theorem 1.10]{M}
	(see also \cite[Proposition 5.1]{DM}) that the equation $($\hyperref[eq:Shubin]{$E_{1,1}$}$)$ is not exactly null-controllable
	in any positive time whenever the control support $\omega \subset \RR^d$ is contained in a half space. In fact, it can be
	readily checked that a half space satisfies a geometric condition of the form
	\begin{equation}\label{eq:densitythick}
		\forall x\in\RR^d,\quad\vert\omega\cap B(x,\rho(x))\vert \geq \theta\vert B(x,\rho(x))\vert,
	\end{equation}
	with a function $\rho:\mathbb R^d\rightarrow\mathbb R_+$ taking the form
	\[
		\rho(x) = L \sprod{x},\quad x \in \RR^d,
	\]
	with some $L > 0$. This raises the question whether local scales $\rho$ can be allowed that exhibit an arbitrary sublinear
	growth. This question has not been answered yet, but a first step in this direction is made by the result
	\cite[Corollary 2.13]{AlphonseS} stating that when the control support $\omega\subset\mathbb R^d$ satisfies the geometric
	condition \eqref{eq:densitythick} with a function $\rho:\mathbb R^d\rightarrow\mathbb R_+$ satisfying
	\[
		\rho(x)
		\leq\frac{L\sprod{x}}{(g\circ g)^\alpha(\abs{x})g(\abs{x})}
		\quad\text{ where }\quad
		g(r) = \log(e+r),\quad r \geq 0,
	\]
	with some $L > 0$ and $\alpha > 2$, then the equation $($\hyperref[eq:Shubin]{$E_{1,1}$}$)$ is exactly null-controllable from
	the control support $\omega$ in every positive time $T>0$. By contrast, Corollary~\ref{cor:exactcontshubin} shows that the
	cost-uniform approximate null-controllability properties of the equation $($\hyperref[eq:Shubin]{$E_{1,1}$}$)$ are far more
	simple since they hold in every positive time $T>0$ and from every measurable control support $\omega\subset\mathbb R^d$ with
	positive measure.
\end{rk}

\subsection{Other models}

The results presented in this section for the anisotropic Shubin semigroups can be generalized to any other strongly continuous
semigroup $(e^{-tA})_{t\geq0}$ acting on $L^2(\mathbb R^d)$ and also satisfying Gelfand--Shilov smoothing and localizing
properties of the form \eqref{eq:smoothshubin} and \eqref{eq:locshubin}. Rather than writing a very general result, we just
present an example different from the anisotropic Shubin operator.

\begin{ex}\label{ex:anharmonic} Let us consider the anharmonic oscillator
	\[
		H = -\Delta + i\vert x\vert^2,\quad x\in\mathbb R^d.
	\]
	The operator $H$ equipped with the domain
	\[
		D(H) = \big\{g\in L^2(\mathbb R^d) : Hg\in L^2(\mathbb R^d)\big\}
	\]
	is known to generate a strongly continuous semigroup $(e^{-tH})_{t\geq0}$ on $L^2(\mathbb R^d)$. Moreover, it follows from the
	work \cite[Theorem 2.6]{AB2} that there exist positive constants $c,c_1>0$ and $t_0\in(0,1)$ such that for all
	$\alpha\in\mathbb N^d$, $t\in(0,t_0)$, and $g\in L^2(\mathbb R^d)$,
	\[
		\big\Vert\partial^{\alpha}_x(e^{-tH}g)\big\Vert_{L^2(\mathbb R^d)}
		\le
		\frac{c^{1+\vert\alpha\vert}}{t^{\frac{\vert\alpha\vert}2}}\,\sqrt{\alpha!}\,\Vert g\Vert_{L^2(\mathbb R^d)},
	\]
	and
	\[
		\big\Vert e^{c_1t^3\vert x\vert^2}(e^{-tH}g)\big\Vert_{L^2(\mathbb R^d)}
		\le
		c\Vert g\Vert_{L^2(\mathbb R^d)}.
	\]
	We thus deduce exactly as in the proof of Theorem~\ref{thm:ucShubinDecay}, but with
	\[
		D = c\norm{g}_{L^2(\RR^d)}
		,\quad
		c_x = \frac{c}{\sqrt{t}}
		,\quad
		\mu = \frac{1}{2}
	\]
	and
	\[
		c_0 = c_1t^3
		,\quad
		\nu = \frac{1}{2}
		,
	\]
	that for every measurable set $\omega\subset\mathbb R^d$ with positive measure, there exists some positive constant $K>0$ such
	that for all $t\in(0,t_0)$, $\scrit>\scrit_{\omega}$, and $g\in L^2(\mathbb R^d)$,
	\[
		\big\Vert e^{-tH}g\big\Vert^2_{L^2(\mathbb R^d)}
		\le
		\bigg(\frac K{\theta_{\scrit}}\bigg)^{K(1+t^3\scrit^2+\scrit^2/t)}\big\Vert e^{-tH}g\big\Vert^2_{L^2(\omega)} +
			2c^2e^{-2c_1t^3\scrit^2}\Vert g\Vert^2_{L^2(\mathbb R^d)},
	\]
	where the radius $\scrit_{\omega}\in[0,+\infty)$ is as in \eqref{eq:distancezero} and the ratio $\theta_{\scrit}\in(0,1]$ is as
	in \eqref{eq:ratio}.
\end{ex}

\appendix

\section{A local estimate with anisotropic smoothing}\label{app:local}

We have the following variant of Corollary~\ref{cor:localEstimatewChangofVariables}.

\begin{prop}\label{prop:localEstimatewChangofVariablesAnisotropic}
	Let $Q = \Psi((0,1)^{2d})$, where the mapping $\Psi \colon \RR^{2d} \to \RR^{2d}$ is given by $\Psi(z) = z_0 + (\ell_x,\ell_y)z$,
	$z \in \RR^{2d}$, with some $z_0 \in \RR^{2d}$ and $\ell_x,\ell_y \in (0,\infty)^d$. Moreover, suppose that $f \in C^\infty(Q)$
	satisfies
	\[
		\forall\alpha,\beta \in \NN^d
		,\quad
		\norm{\partial_x^\alpha \partial_y^\beta f}_{L^2(Q)}
		\leq
		A B_x^{\alpha} B_y^\beta (\abs{\alpha}!)^{\mu_x} (\abs{\beta}!)^{\mu_y}  \norm{f}_{L^2(Q)}
	\]
	with some $A > 0$, $B_x,B_y \in (0,\infty)^d$, and $\mu_x,\mu_y \in [0,1)$. Then, there is a constant $K = K(d,\mu_x,\mu_y) > 0$,
	depending only on $\mu_x$, $\mu_y$, and the dimension $d$, such that for every measurable set $E \subset Q$ of positive measure
	we have
	\[
		\norm{f}_{L^2(Q)}^2
		\leq
		\biggl( \frac{K\vert Q\vert}{\vert E\vert} \biggr)^{KC}
			\norm{f}_{L^2(E)}^2
	\]
	with
	\[
		C
		=
		1 + \log(A) + (\abs{\ell_x B_x})^{\frac{1}{1-\mu_x}} + (\abs{\ell_y B_y})^{\frac{1}{1-\mu_y}}
		.
	\]
\end{prop}

For the proof of Proposition~\ref{prop:localEstimatewChangofVariablesAnisotropic}, we rely on the following particular case of
\cite[Proposition~31]{EgidiS-21} tailored to our situation.

\begin{prop}\label{prop:EgidiS}
	Suppose that $g \in W^{\infty,2}((0,1)^d) = \bigcap_{k\in\NN} W^{k,2}((0,1)^d)$ satisfies
	\[
		\forall m \in \NN
		,\quad
		\sum_{\abs{\alpha}=m} \frac{1}{\alpha!} \norm{ \partial^\alpha g }_{L^2((0,1)^d)}^2
		\leq
		\frac{C(m)}{m!} \norm{g}_{L^2((0,1)^d)}^2
	\]
	with constants $C(m)> 0$ such that
	\[
		h
		:=
		\sum_{m\in\NN} \sqrt{C(m)} \frac{(10d)^m}{m!}
		<
		\infty
		.
	\]
	Then, there is a constant $K = K(d) > 0$, only depending on the dimension $d$, such that for every measurable set
	$\cE \subset (0,1)^d$ of positive measure we have
	\[
		\norm{g}_{L^2((0,1)^d)}^2
		\leq
		\biggl( \frac{K}{\vert\cE\vert} \biggr)^{K(1+\log h)} \norm{g}_{L^2(\cE)}^2
		.
	\]
\end{prop}

\begin{proof}[Proof of Proposition~\ref{prop:localEstimatewChangofVariablesAnisotropic}]
	We may assume that $f$ is not identically equal to zero. Moreover, it is easy to see that the function
	$g := f \circ \Psi \in W^{\infty,2}((0,1)^{2d})$ satisfies
	\[
		\norm{ \partial_x^\alpha \partial_y^\beta g }_{L^2((0,1)^{2d})}
		\leq
		A (\ell_xB_x)^\alpha (\ell_yB_y)^\beta (\abs{\alpha}!)^{\mu_x} (\abs{\beta}!)^{\mu_y}  \norm{g}_{L^2((0,1)^{2d})}
		.
	\]
	Taking into account that
	\[
		\sum_{\abs{\alpha}=k} \frac{(\ell_xB_x)^{2\alpha}}{\alpha!}
		=
		\frac{\abs{\ell_xB_x}^{2k}}{k!}
		\quad\text{ and }\quad
		\sum_{\abs{\alpha}=m-k} \frac{(\ell_yB_y)^{2\alpha}}{\alpha!}
		=
		\frac{\abs{\ell_yB_y}^{2(m-k)}}{(m-k)!}
	\]
	for $k \in \NN$ with $0 \leq k \leq m$, we observe
	\begin{multline*}
		m! A^2\sum_{\abs{(\alpha,\beta)}=m} \frac{1}{\alpha! \beta!} (\ell_xB_x)^{2\alpha} (\ell_yB_y)^{2\beta} (\abs{\alpha}!)^{2\mu_x}
			(\abs{\beta}!)^{2\mu_y}\\
		\leq
		A^2\sum_{k=0}^m \binom{m}{k} (k!)^{2\mu_x}\abs{\ell_xB_x}^{2k} ((m-k)!)^{2\mu_y}\abs{\ell_yB_y}^{2(m-k)}
		\leq
		C(m)
	\end{multline*}
	with
	\[
		C(m)
		:=
		\left( A\sum_{k=0}^m \binom{m}{k} (k!)^{\mu_x}\abs{\ell_xB_x}^{k} ((m-k)!)^{\mu_y}\abs{\ell_yB_y}^{(m-k)} \right)^2
		.
	\]
	Thus, the function $g$ satisfies
	\begin{equation}\label{eq:Bernstein}
		\forall m\in\NN
		,\quad
		\sum_{\abs{(\alpha,\beta)}=m} \frac{1}{\alpha! \beta!} \norm{ \partial_x^\alpha \partial_y^\beta g }_{L^2((0,1)^{2d})}^2
		\leq
		\frac{C(m)}{m!} \norm{g}_{L^2((0,1)^d)}^2
		.
	\end{equation}

	Now, the asymptotics
	\[
		\sum_{m = 0}^\infty\frac{t^m}{(m!)^p}
		=
		\frac{\euler^{pt^{1/p}}}{p^{1/2}(2\pi t^{1/p})^{(p-1)/2}}\Bigl\{ 1 + O\Bigl( \frac{1}{t^{1/p}} \Bigr) \Bigr\}
		\qquad
		(p \in (0,4],\ t \to \infty)
	\]
	derived in \cite[Chapter 8, Eq.~(8.07)]{Olver-97} imply that for every $\mu \in (0,1]$ there is a constant $K_\mu > 0$ such that
	\[
		\forall t \geq 0
		,\quad
		\sum_{m=0}^\infty \frac{t^m}{(m!)^{1-\mu}}
		\leq
		K_\mu\euler^{t^{1/(1-\mu)}}
		.
	\]
	Setting $K' := \max\{K_{\mu_x},K_{\mu_y}\}$, by Cauchy product formula this yields
	\begin{align*}
		h
		:=
		\sum_{m=0}^\infty \sqrt{C(m)} \frac{(10d)^m}{m!}
		&=
		A \sum_{m=0}^\infty \sum_{k=0}^m \frac{(k!)^{\mu_x}\abs{10d\ell_xB_x}^{k}}{k!}
			\frac{((m-k)!)^{\mu_y}\abs{10d\ell_yB_y}^{(m-k)}}{(m-k)!}\\
		&=
		A \Biggl( \sum_{k=0}^\infty \frac{\abs{10d\ell_xB_x}^{k}}{(k!)^{1-\mu_x}} \Biggr)
			\Biggl( \sum_{l=0}^\infty \frac{\abs{10d\ell_yB_y}^{l}}{(l!)^{1-\mu_y}} \Biggr)\\
		&\leq
		A (K')^2 \exp\Bigl( \abs{10d\ell_xB_x}^{\frac{1}{1-\mu_x}} + \abs{10d\ell_yB_y}^{\frac{1}{1-\mu_y}} \Bigr)
		<
		\infty
		.
	\end{align*}
	In light of \eqref{eq:Bernstein} and the latter, we may apply Proposition~\ref{prop:EgidiS} to $g$ in dimension
	$2d$ instead of $d$ and with the set $\cE := \Psi^{-1}(E) \subset (0,1)^{2d}$ to obtain
	\[
		\frac{\norm{f}_{L^2(Q)}^2}{\norm{f}_{L^2(E)}^2}
		=
		\frac{\norm{g}_{L^2((0,1)^{2d})}^2}{\norm{g}_{L^2(\cE)}^2}
		\leq
		\biggl( \frac{K}{\vert\cE\vert} \biggr)^{K(1+\log h)}
		=
		\biggl( \frac{K\vert Q\vert}{\vert E\vert} \biggr)^{K(1+\log h)}
		,
	\]
	which proves the claim upon a suitable adaptation of the constant $K$.
\end{proof}%

\section{Exponential form of the ultra-analytic regularity}\label{app:expoential}

In this second and last section of the appendix, we give the proof of the following lemma which was used in
Section~\ref{subsec:smoothingbg} to deduce the estimates \eqref{eq:expoformregxbg} from Lemma \ref{lem:regxbg}.

\begin{lem}\label{lem:expoform}
	Let $\Lambda_1>0$, $\Lambda_2>1$ be some positive constants and $m\geq1$ be a positive integer. For every function
	$u\in L^2(\mathbb R^d)$ satisfying
	\begin{equation}\label{eq:gevrey}
		\forall\alpha\in\mathbb N^d
		,\quad
		\Vert\partial_x^{\alpha}u\Vert_{L^2(\mathbb R^d)}
		\leq
		\Lambda_1\Lambda_2^{\vert\alpha\vert}\ (\alpha!)^{\frac1m},
	\end{equation}
	the following estimate holds
	\[
		\big\Vert e^{C_{m,\Lambda_2}\vert D_x\vert^m}u\big\Vert_{L^2(\mathbb R^d)}
		\leq
		2^d\Lambda_1,
	\]
	where the positive constant $C_{m,\Lambda_2}>0$ is given by
	\begin{equation}\label{eq:cstexpo}
		C_{m,\Lambda_2} = \frac1{2em(2d\Lambda_2)^m}.
	\end{equation}
\end{lem}

\begin{proof}
	Let $u\in L^2(\mathbb R^d)$ satisfy \eqref{eq:gevrey}. By using the estimate
	\[
		\forall N\geq0, \forall\xi\in\mathbb R^d,\quad\vert\xi\vert^N\le d^N\sum_{\vert\alpha\vert = N}\vert\xi^{\alpha}\vert,
	\]
	we get that for all $N\geq0$,
	\begin{equation}\label{20032018E1}
		\begin{aligned}
			\big\Vert\vert D_x\vert^Nu\big\Vert_{L^2(\mathbb R^d)}
			& \le d^N\sum_{\vert\alpha\vert = N}\big\Vert\partial^{\alpha}_xu\big\Vert_{L^2(\mathbb R^d)} \\[5pt]
			& \le d^N\sum_{\vert\alpha\vert = N}\Lambda_1 \Lambda_2^{\vert\alpha\vert} (\alpha!)^{\frac1m} \\[5pt]
			& \le d^N2^{N+d-1}\Lambda_1\Lambda_2^N (N!)^{\frac1m},
		\end{aligned}
	\end{equation}
	where we used the fact that
	\[
		\#\big\{\alpha\in\mathbb N^d : \vert\alpha\vert = N\big\} = \binom{N+d-1}{N}\le 2^{N+d-1}.
	\]
	Considering the positive constant $C_{m,\Lambda_2}>0$ defined in \eqref{eq:cstexpo}, we deduce from \eqref{20032018E1} that
	\begin{align*}
		\big\Vert e^{C_{m,\Lambda_2}\vert D_x\vert^m}u\big\Vert_{L^2(\mathbb R^d)}
		& \le \sum_{N=0}^{+\infty}\frac1{2^N}\frac1{(2d\Lambda_2)^{mN}}\frac1{(em)^NN!}\big\Vert\vert D_x\vert^{mN}u\big\Vert_{L^2(\mathbb R^d)} \\
		& \le 2^{d-1}\Lambda_1 \sum_{N=0}^{+\infty}\frac1{2^N}\frac{((mN)!)^{\frac1m}}{(em)^NN!}. \nonumber
	\end{align*}
	Moreover, by using the fact that $N^N\le e^NN!$ for every $N\geq0$, we get that
	\[
		((mN)!)^{\frac1m}\le(mN)^N\le(me)^NN!.
	\]
	The proof of Lemma \ref{lem:expoform} is now ended.
\end{proof}

\end{document}